\documentclass{amsart}
\usepackage{amssymb,latexsym}
\usepackage{color}
\usepackage{graphicx}
\usepackage[mathcal]{euscript}
\usepackage{mathrsfs}
\usepackage{enumerate}

\makeatletter

\newtheorem{thm}{Theorem}[section]
\newtheorem{cor}[thm]{Corollary}
\newtheorem{lem}[thm]{Lemma}

\newtheorem{prop}[thm]{Proposition}
\newtheorem{que}{Question}

\theoremstyle{definition}

\newtheorem{rem}[thm]{Remark}
\newtheorem{exmp}[thm]{Example}

\numberwithin{equation}{section}

\DeclareMathOperator{\alp}{alph}
\DeclareMathOperator{\dist}{dist}
\DeclareMathOperator{\diam}{diam}
\DeclareMathOperator{\Tran}{Trans}
\DeclareMathOperator{\Int}{int}
\DeclareMathOperator{\Prox}{Prox}
\DeclareMathOperator{\SProx}{SyProx}

\DeclareMathOperator{\Asy}{Asy}
\DeclareMathOperator{\M}{M}
\DeclareMathOperator{\Rec}{Rec}

\newcommand{\Al}{A}

\newcommand{\N}{\mathbb{N}}
\newcommand{\C}{\mathbb{C}}

\newcommand{\Si}{{S^1}}
\newcommand{\Z}{\mathbb{Z}}
\newcommand{\R}{\mathbb{R}}
\newcommand{\orbp}{\mbox{O}^+}

\newcommand{\set}[1]{\left\{#1\right\}}
\newcommand{\eps}{\varepsilon}
\newcommand{\ra}{\rightarrow}

\newcommand{\htop}{{h_{\text{top}}}}

\begin{document}

\title{Syndetic proximality and scrambled sets}

\author{T.K. Subrahmonian Moothathu}
\address[T.K.S. Moothathu]{Department of Mathematics and Statistics, University of Hyderabad, Hyderabad 500046, India}\email{tksubru@gmail.com}
\author{Piotr Oprocha}
\address[P. Oprocha]{AGH University of Science and Technology\\
Faculty of Applied Mathematics\\
al. A. Mickiewicza 30, 30-059 Krak\'ow,
Poland\\ -- and --\\Institute of Mathematics\\ Polish Academy of Sciences\\ ul. \'Sniadeckich 8, 00-956 Warszawa, Poland} \email{oprocha@agh.edu.pl}

\begin{abstract}
This paper is a systematic study about the syndetically proximal relation and the possible existence of syndetically scrambled sets for the dynamics of continuous self-maps of compact metric spaces. Especially we consider various classes of transitive subshifts, interval maps, and topologically Anosov maps. We also present many constructions and examples.
\end{abstract}
\maketitle

\tableofcontents

\section{Introduction}

Events that happen infinitely often with bounded gaps in time are more interesting than events that just happen infinitely often. Hence the tag \emph{syndetic} has prominence in the theory of dynamical systems.

The main objects of study in this article are syndetically proximal pairs and syndetically scrambled sets, that is, scrambled sets in which every pair is syndetically proximal. The definition of syndetically proximal relation was available in the literature more than forty years ago \cite{Clay}, but not much has been proved beyond the results contained in the book of Glasner \cite{GlasnerProx} (some facts on syndetically proximal relation can be obtained as a byproduct of very abstract and general results, e.g. results on actions of groups, where proximality is studied in terms of Furstenberg families, etc.; particularly from \cite{Shao}, or \cite{AkinB1,AAG}).
This is a little bit surprising given that a lot has been done in the context of scrambled sets (see the recent survey \cite{SnohaSurv}). On the other hand, we must also point out that Kuratowski-Mycielski Theorem, which is one of the main tools in the theory of scrambled sets, is not so useful while dealing with the syndetically proximal relation since most often this relation is a first category subset of $X\times X$.

Here we aim to fill the gap in the literature by developing various tools that can be used for the construction of syndetically scrambled sets possessing various interesting properties (as Cantor sets, dense Mycielski sets, invariant sets, etc.). Our starting point is the recent work \cite{tksm} of one of the authors, where a few properties of the syndetically proximal relation were studied. These studies will be widely extended here, providing a nearly self-contained toolbox for further research. Especially we will examine the possible existence of syndetically scrambled sets for various classes of transitive subshifts, interval maps and topologically Anosov maps. We will also show that in some cases syndetically scrambled sets can be transferred via factor maps, so that syndetically scrambled sets can be detected in systems arising in applications (for example, in systems with a local product structure).

The paper is organized as follows. In the next section we provide most of the definitions used later. The reader may skip this section for the moment and refer to it when necessary.
Section~\ref{sec:tools} collects many practical tools that are useful for the construction of syndetically scrambled sets; in particular, they will be applied later in various proofs in the paper. We provide techniques which (under some mild assumptions) help us to change a syndetically scrambled set from an uncountable set to a Cantor set, from a Cantor set to a dense Mycielski set, etc., and also to transfer such a set through a countable-to-one factor map.

Next, we focus on more concrete classes of dynamical systems. In section~\ref{sec:transSD} we consider various classes of transitive subshifts: shifts of finite type, synchronizing subshifts (in particular all subshifts with the specification property) and coded subshifts (this class contains the so-called $\beta$-shifts). Section~\ref{sec:minimal} is devoted to the study of syndetically proximal pairs in minimal subshifts, with main emphasis on those arising from substitutions. We show that for substitutional subshifts, in contrast to the case of Toeplitz subshifts, the set of proximal pairs can be strictly larger than the set of syndetically proximal pairs. Section \ref{CounterEg} presents the construction of an example where the syndetically proximal relation is trivial but the proximal relation is non-trivial.

In section~\ref{sec:interval} we analyze maps of the unit interval from the point of view of syndetically proximal relation. We show that interval maps with zero entropy do not distinguish between proximal and syndetically proximal pairs, and while in the case of transitive maps on the interval these two relations are quite different (proximal pairs form a residual subset of $[0,1]^2$, while syndetically proximal pairs are in the complement of a residual set), there always exists quite large syndetically scrambled set for such a map.

The present paper is by no means complete characterization of syndetically proximal relation and syndetically scrambled sets. Many interesting questions
remain still open (e.g. we still do not know any answer to the main questions of \cite{tksm}). The authors hope that the results included in the paper will be a good motivation for other mathematicians to study further properties of syndetically proximal relation, and that the answers to these questions can be obtained in the future.

\section{Preliminaries}

Denote by $\N$ ($\mathbb{Z}, \mathbb{R}$, respectively) the set of all positive integers (integers, real numbers, respectively).
The set of non-negative integers is denoted $\N_0$, that is $\N_0=\N\cup \set{0}$.
Let $A\subset \N$. We say that $A$ is: \emph{thick} if for any $n\in \N$ there is $i$ such that $\set{i,i+1,\ldots, i+n}\subset A$; \emph{syndetic} if it intersects every thick set; \emph{co-finite} is $\N\setminus A$ is finite; \emph{thickly syndetic} if for every $n\in \N$ the set $\set{i\; : \:\set{i,i+1,\ldots,n}\subset A}$ is syndetic; \emph{piecewise syndetic} if it intersects every thickly syndetic set.

Let $(X,d)$ be a compact metric space. A subset is \emph{residual} if it contains a dense $G_\delta$ set, and \emph{first category} if it is in complement of a residual set. We say that a set $S\subset X$ is \emph{perfect} if it is closed without isolated points. We say that a non-empty set $X$ is \emph{totally disconnected} if all connected components of $X$ are singletons. By a \emph{Cantor set} we mean a perfect totally disconnected set.
The following two types of subsets $A\subset X$ will be of particular interest:
\begin{itemize}
\item[--] \emph{c-dense sets}, that is sets $A$ so that $A\cap U$ is uncountable for any nonempty open set $U\subset X$,
\item[--] \emph{Mycielski sets}, that is sets which are countably infinite unions of Cantor sets.
\end{itemize}
Note that every dense Mycielski set is also c-dense. From the other side, every uncountable Borel subset of a compact metric space contains a Cantor set \cite[Theorem~3.2.7]{Sriv}, and so every Borel c-dense subset contains a dense Mycielski subset.

\subsection{Topological dynamics}
By a \emph{dynamical system} we mean a pair $(X,f)$ where $X$ is a compact metric space and $f\colon X \ra X$ is a continuous map. When $f$ is a homeomorphism, we say that $(X,f)$ is an \emph{invertible dynamical system}. If $(X,f)$ and $(Y,g)$ are dynamical systems then by $(X\times Y, f\times g)$ we denote the product dynamical systems, where $X\times Y$ is endowed with any metric defining the product topology on $X\times Y$.

Let $(X, f)$ be a dynamical system and $A, B$ be non-empty subsets of $X$. Put $N (A, B)= \{n\in
\mathbb{N}: f^n (A)\cap B\neq \emptyset\}$. Observe that $N (A, B)= \{n\in \mathbb{N}: A\cap f^{- n} (B)\neq \emptyset\}$.
When $A=\set{x}$, we will simply write $N(x,B)$ instead of $N(\set{x},B)$.

A point $x$ is: a \emph{fixed point} if $f(x)=x$; \emph{periodic} if $f^n(x)=x$ for some $n\in \N$; \emph{recurrent} if $N(x,U)\neq \emptyset$ for any open set
$U\ni x$. When $x$ is a periodic point, the smallest number $n\in \N$ such that $f^n(x)=x$ is called the \emph{prime period} of $x$. The set of recurrent points will be denoted by $\Rec(f)$.
A point $y\in X$ is an \emph{$\omega$-limit
point} of a point $x$ if it is an accumulation point of the sequence
$x,f(x),f^2(x),\dots$. The set of all $\omega$-limit points of $x$
is called the \emph{$\omega$-limit set} of $x$ or \emph{positive limit set of $x$} and is denoted by
$\omega(x,f)$. The positive orbit of a point $x$ is the set $\orbp(x,f)=\set{x,f(x),f^2(x),\ldots}$.

If $M\subset X$ is nonempty, closed, invariant (i.e. $f(M)\subset M$) and has no proper subset with these three
properties then we say that $M$ is \emph{a minimal set} for $f$. If $X$ is the minimal set for $f$ then we say that $f$ is a \emph{minimal map} and $(X,f)$ is a \emph{minimal system}.
Elements of a minimal system are called \emph{minimal points} or \emph{uniformly recurrent points}. It is well know that for any minimal system $M$, $\omega(x,f)=M$ for every
$x\in M$; and $N(x,U)$ is syndetic for every $x\in M$ and every neighborhood $U$ of $x$. The set of all minimal points of $f$ is denoted as $\M(f)$.

We say that a dynamical system $(X,f)$ (or $f$ for short) is: \emph{transitive} if $N(U,V)\neq \emptyset$ for any two nonempty open sets $U,V\subset X$; \emph{totally transitive} if $(X,f^n)$ is transitive for any $n\in \N$; \emph{weakly mixing} if $(X\times X, f\times f)$ is transitive; \emph{mixing} if $N(U,V)$ is co-finite for any two nonempty open sets $U,V\subset X$; \emph{exact} if for every nonempty open set $U\subset X$ there is $n\in \N$ such that $f^n(U)=X$.
The set of \emph{transitive points}, that is points with dense orbits, is denoted by $\Tran(f)$.

We say that a map $\pi \colon X \ra Y$ is \emph{a factor map} or \emph{a semi-conjugacy between dynamical systems $(X,f)$, $(Y,g)$}, denoted by $\pi\colon (X,f)\ra (Y,g)$,
if $\pi$ is a continuous surjection and $\pi \circ f = g \circ \pi$. Sometimes we will also say that $(Y,g)$ is
a \emph{factor} of $(X,f)$ or that $(X,f)$ is an \emph{extension} of $(Y,g)$. When $\pi$ is a homeomorphism, we say that $\pi$ is \emph{a conjugacy} or
that $(X,f), (Y,g)$ are \emph{conjugate}.

\subsection{Scrambled sets and entropy}

Let $(X,f)$ be a dynamical system and let $d$ be an admissible metric on $X$. The \emph{asymptotic}, \emph{proximal}, and \emph{syndetically proximal} relations for $f$, are defined respectively as

\begin{eqnarray*}
\Asy(f)&=&\set{(x,y)\in X^2:\lim_{n\to \infty}d(f^n(x),f^n(y))=0},\\
\Prox(f)&=&\set{(x,y)\in X^2:\liminf_{n\to \infty}d(f^n(x),f^n(y))=0},\\
\SProx(f)&=&\Big\{(x,y)\in X^2:\\
&&\quad\quad\set{n\in \N:d(f^n(x),f^n(y))<\eps}\text{ is syndetic for all } \eps>0\Big\}.
\end{eqnarray*}
If $R\subset X\times X$ is a symmetric relation and $x\in X$ then we will usually denote the cell of $x$ in $R$ by $R(x)=\set{y: (x,y)\in R}$.

We say that $(x,y)$ is a \emph{scrambled pair} for $f$ if $(x,y)\in \Prox(f)\setminus \Asy(f)$, and $(x,y)$ is a \emph{syndetically scrambled pair} for $f$ if $(x,y)\in \SProx(f)\setminus \Asy(f)$. Note that $(x,y)$ is a \emph{(syndetically) scrambled pair} for $f$ iff $(x,y)$ is a \emph{(syndetically) scrambled pair} for $f^k$ for every $k\in \mathbb N$. A subset $S\subset X$ containing at least two points is a \emph{(syndetically) scrambled set} for $f$ if $(x,y)$ is a \emph{(syndetically) scrambled pair} for $f$ for any two distinct $x,y\in S$. Moreover, if there is $\eps>0$ such that $\limsup_{n\to \infty}d(f^n(x),f^n(y))\ge \epsilon$ for any two distinct $x,y\in S$ we say that $S$ is $\eps$-scrambled or syndetically $\eps$-scrambled respectively.
We will say that a scrambled set $S\subset X$ is \emph{invariant} for $(X,f)$ if $f(S)\subset S$.
Note that in the case of invertible
dynamical systems this definition does not coincide with the standard meaning of invariant set, since in this setting it is usually required that an invariant set $A$ satisfies $f(A)=A$.
To avoid any unambiguity, we will call sets with the property $f(A)=A$ \emph{strongly invariant}.

A dynamical system $(X,f)$ is said to be \emph{Li-Yorke chaotic} if it has an uncountable scrambled set.

For a definition of the \emph{topological entropy} of $f$ we refer the
reader to \cite{ALM}. that, if $X$ is a compact space then
the entropy of $f$ is a (possibly infinite) number $\htop(f)\in [0,+\infty]$.
We will use the basic properties of the entropy such as those in
\cite[Section 4.1]{ALM} without further reference.
Recently it was proved that positive topological entropy implies Li-Yorke chaos \cite{PTELY}.

We say that a point $x$ is \emph{distal}, if $(x,y)\notin \Prox(f)$
for every $y\in\omega(x,f)\setminus \set{x}$. A dynamical system $(X,f)$ is distal if every $x\in X$ is distal for $f$.
It is known that every point is proximal to some minimal point \cite{Furstenberg} (this statement is nontrivial when given point is not minimal), in particular distal points are always minimal and every distal system decomposes into a sum of minimal systems. It is also clear from the above that distal system can be defined equivalently as a system without proper (i.e. non-diagonal) proximal pairs.

\subsection{Symbolic dynamics}

Let $\Al$ be a finite set (\emph{an alphabet}) and denote by $\Al^*$ the set of all finite words over $\Al$.
The set $\Al^{*}$
with the concatenation of words is a free monoid with the minimal set of
generators $\Al$ (we assume that the empty word, denoted
by $\lambda$ is in $\Al^*$). The set of nonempty words is denoted by
$\Al^+$. If $w\in \Al^+$ and $k\in \N$ then by $w^k=ww\cdots w$ we denote the $k$-times concatenation of $w$ with itself.
If $w=a_0a_1\cdots a_{k-1} \in \Al^+$ then we write $w_i:=a_i$ or $w[i]:=a_i$; the number $|w|=k$ is said to be \emph{the length of $w$}; by $|w|_a$ is denoted the number of occurrences of the letter $a$ in
$w$.

If $x\in \Al^{\N_0}$ and $i\leq j$
are integers then we denote $x_{[i,j]}=x_i x_{i+1}\dots x_j$ and $x_{[i,j)}=x_{[i,j-1]}$. If $i>j$ then $x_{[i,j]}=\lambda$, where $\lambda$ is the empty word.
If $x= y$ then we put $d(x,y)=0$ and $d(x,y)=2^{-k}$ otherwise, where $k\geq 0$ is the maximal number such that $x_{[0,k)}=y_{[0,k)}$. We also define the shift map $\sigma\colon \Al^{\N_0}\ra \Al^{\N_0}$ by putting $\sigma(x)_i=x_{i+1}$ for all $i\in \N$.
It can be easily verified that $(\Al^{\N_0},d)$ is a compact metric space and $\sigma$ is continuous.
For any $w\in \Al^+$ we define its \emph{cylinder set} $C[w]:=\set{x\in \Al^{\N_0}\; ; \; x_{[0,n)}=w}$ where $n=|w|$. It is
well know that cylinder sets form a basis of the topology of $(\Al^{\N_0},d)$. By \emph{full  (one-sided) shift} over the alphabet $\Al$ we mean the pair $(\Al^{\N_0},\sigma)$.
If $\Al=\set{0,1,\ldots,m-1}$ for some $m\in \N$ then we say that $(\Al^{\N_0},d)$ is full (one-sided) shift on $m$ symbols. For simplicity we denote the full shift on $m$ symbols by $(\Sigma^+_m,\sigma)$, or simply $\Sigma_m^+$, where obviously $\Sigma^+_m=\set{0,\ldots,m-1}^{\N_0}$.

If $X$ is nonempty, closed, and $\sigma$-invariant (i.e. $\sigma(X)\subset X$), then the dynamical system $(X,\sigma|_X)$ is called a \emph{shift} or a \emph{subshift}. For simplicity, we will often write $X$ instead of the pair $(X,\sigma|_X)$, and when $X$ is clear from the context we will also write $\sigma$ instead of $\sigma|_X$.
For simplicity, we write $C_X[w]=C[w]\cap X$ where $X$ is a subshift.
By $L(X)$ we denote \emph{the language} of subshift $X$, that is the set $L(X):=\set{x_{[0,k]}\; : \; x\in X, k\geq 0}$.

Similarly, we can define two-sided subshifts of $\Al^\Z$ with the only difference that in the definition of metric, the integer $k$ is the maximal number such that $x_{(-k,k)}=y_{(-k,k)}$. In that case all other definitions presented above are modified accordingly. In particular, two-sided full shift on $m>0$ symbols will be denoted by $\Sigma_m$.

Let $F\subset \Al^+$ and let $X_F=\set{x\in \Al^{\N_0}\; : \; x_{[i,i+n]}\not\in F \text{ for all } i,n\geq 0}$. We say that a shift is a \emph{shift of finite type} or simply SFT if $X=X_F$ for a finite set of words $F$. Recall that a subshift $X$ is \emph{sofic} if it is a factor of a SFT.

Let $(X,\sigma)$ be a subshift and let $L(X)$ be its language. A nonempty word $s\in L(X)$ is said to be a \emph{synchronizing word} if we have $usv\in L(X)$ whenever $us,sv\in L(X)$. We say after \cite{BlaHan} that a subshift $(X,\sigma)$ is a \emph{synchronizing subshift} if it is transitive and has a synchronizing word. By well known facts on graph presentations of sofic shifts (in particular, properties of so-called follower sets of vertices in these presentations \cite{LM}) it is immediate to see that every transitive sofic shift is synchronizing.
For a given set of words $W\subset \Al^+$, let $X_W$ be the smallest
subshift containing all possible bi-infinite sequences obtained by a concatenation of words in $W$. We
call $X_W$ a \emph{coded system} or \emph{coded subshift}. If $X$ is a synchronizing subshift with synchronizing word $s$, then for any $u\in L(X)$ there are $v,v'$ such that $svuv's\in L(X)$. Then we see that any synchronizing subshift is coded system generated by the set $W=\set{us : sus\in L(X)}$. This was first proved in \cite{BlaHan}.
In fact we have the following relation between various classes of mixing subshifts (for the definition of the specification property see \cite{BowenSpec} or \cite{Denker}):

$$
\text{SFT } \Longrightarrow \text{ sofic }\Longrightarrow \text{ specification }\Longrightarrow\text{ synchronizing }\Longrightarrow\text{ coded}.
$$

Recall that a collection of words $B_1,\ldots, B_k\subset \Al^+$ is a \emph{circular code} if for any sequence $x_1,\ldots, x_m$, $y_1,\ldots, y_n\in \set{B_i}$
and any $p,s\in \Al^*$ the equalities $sp=x_1$ and $p x_2 \ldots x_m s=y_1\ldots y_n$ imply that $m=n$, $p=s=\lambda$ and $x_i=y_i$ for $i=1,\ldots,m$.
It is known (e.g. see \cite[Proposition 29 ]{BP}) that every coded system generated by a circular code is an SFT.

It was probably first proved by Krieger (see \cite{BlaHan2} for more details) that if $X_W$ is a coded system, then there exists a sequence of words $\mathcal{B}=\set{B_1,B_2,\ldots}$
such that for every $n\in \N$ the list $B_1,\ldots, B_n$ is a circular code and $X_W=X_{\mathcal{B}}$. In particular, there exists an increasing sequence $X_1\subset X_2 \subset \cdots \subset X_{\mathcal{B}}$ of transitive SFT such that $X_W=\overline{\bigcup_{i=1}^\infty X_i}$.

\begin{rem}
When $(X,\sigma)$ is a subshift then it is very easy to verify that any scrambled set is also $1$-scrambled. So we will only say `scrambled' instead of `$\eps$-scrambled' while discussing subshifts.
\end{rem}

\section{A few useful tools}\label{sec:tools}

The method of construction of Mycielski sets by the use of residual relations originated from the works of Mycielski \cite{Mycielski} and Kuratowski \cite{Kur73}. The following version of Mycielski's Theorem \cite{PTELY} will be useful in some cases (see \cite{Akin, SnohaSurv} for more details on applications of this technique).

\begin{thm}[Kuratowski-Mycielski]\label{thm:RelMycielski}
Let $X$ be perfect and $R_n\subset X^n$ be residual in $X^n$ for $n\in \N$. Then there is a dense Mycielski set $S\subset X$ with the property that $(x_1,\ldots,x_n)\in R_n$ for any $n\in \N$ and any pairwise distinct points $x_1,\ldots,x_n\in S$.
\end{thm}

\begin{rem}
Using Kuratowski-Mycielski Theorem, it was proved in \cite{SnohaSurv} that there exists an uncountable $\eps$-scrambled set exactly when there exists
a Cantor $\eps$-scrambled set. Such a characterization for scrambled sets is still an open problem.
\end{rem}

It is known that points with dense orbit form a residual subset of the underlying space for any transitive map. Then as an immediate corollary of Kuratowski-Mycielski Theorem we get the following (see \cite{Iwanik}).

\begin{cor}\label{wMixScra}
Let $(X,f)$ be a non-trivial weakly mixing system and let $f^{\times n}\colon X^n\to X^n$ be the $n$-fold product map. Then there is a dense Mycielski set $S\subset X$ with the property that $(x_1,\ldots,x_n)\in \Tran(f^{\times n})$ for any $n\in \N$ and any pairwise distinct points $x_1,\ldots,x_n\in S$. In particular, $S$ is $\eps$-scrambled for any positive $\eps\le \diam[X]$.
\end{cor}
\begin{proof}
This is a direct application of Kuratowski-Mycielski Theorem since $X$ is perfect and $\Tran(f^{\times n})$ is residual in $X^n$ when $f$ is weakly mixing.
\end{proof}

\begin{rem}
The Kuratowski-Mycielski Theorem, while very useful in many situations, has very limited applications in the context of syndetically proximal pairs.
Simply, by \cite[Corollary 1]{tksm} the relation $\SProx(f)$ is not residual in most cases. In fact we have the following property, complementing Corollary \ref{wMixScra}.
\end{rem}

\begin{prop}
Let $(X,f)$ be a weakly mixing system with at least two minimal points in it. Then $\SProx(f)$ is of first category in $X^2$. Consequently, there is a dense Mycielski set $S\subset X$ such that
\begin{enumerate}[(i)]
\item $(x,y)\in Trans(f^{\times 2})\setminus \SProx(f)$ for any two distinct $x,y\in S$.

\item $S$ is $\eps$-scrambled for $f$ for any positive $\eps\le \diam [X]$.
\end{enumerate}
\end{prop}
\begin{proof}
By \cite[Corollary 1]{tksm} we have that $\SProx(f)$ is a first category set disjoint with $\Tran(f^{\times 2})$. Therefore (i) follows by Kuratowski-Mycielski Theorem and (ii) is a consequence of (i).
\end{proof}

Motivated by the above facts, we will try to answer (at least partially) the following questions:

\begin{que}
How can we prove the existence of syndetically scrambled sets?
\end{que}

\begin{que}
If we know that there is an uncountable syndetically $\eps$-scrambled set, can we establish the existence of a Cantor syndetically $\eps$-scrambled set, at least in some cases?
\end{que}

\begin{que}
If we know that there is a Cantor syndetically scrambled set, can we establish the existence of a dense Mycielski syndetically scrambled set, at least in some cases?
\end{que}

\subsection{From uncountable to Cantor}

We start with some simple observations about subshifts, related to the first two questions.

\begin{prop}\label{onesideSy}
Let $(X,\sigma)$ be a one-sided subshift. If there is $z\in X$ whose syndetically proximal cell $\{y\in X:(y,z)\in \SProx(\sigma)\}$ is uncountable, then $(X,\sigma)$ has a Cantor syndetically scrambled set.
\end{prop}
\begin{proof}
Since $\{y\in X:(y,z)\in \SProx(\sigma)\}$ is an uncountable Borel set (see \cite[Proposition~5]{tksm}), it contains a Cantor set $K$. Consider the asymptotic equivalence relation $x\sim y$ iff $(x,y)\in \Asy(\sigma)$ on $K$. This is an $F_\sigma$ relation since $\{(x,y)\in K^2:x\sim y\}=\bigcup_{k\in \mathbb N}\bigcap_{n=k}^\infty \{(x,y)\in K^2:x_n=y_n\}$. Also each equivalence class of the relation $\sim$ is countable since $(X,\sigma)$ is a one-sided subshift. Hence by Theorem 5.13.9 of \cite{Sriv}, there is a Cantor set $S\subset K$ which intersects each equivalence class of $\sim$ in at most one element. Clearly $S$ is syndetically scrambled for $\sigma$ which ends the proof.
\end{proof}

\begin{cor}
Let $(X,\sigma)$ be a one-sided subshift possessing an uncountable syndetically scrambled set. Then $(X,\sigma)$ has a Cantor syndetically scrambled set.
\end{cor}

If $(X,\sigma)$ is a two-sided subshift, then the above technique will not work. For instance it can happen that there is $z\in X$ such that $\{y\in X:y_{[0,\infty)}=z_{[0,\infty)}\}$ is uncountable. However, this problem can be resolved if we also demand asymptoticity under the action of $\sigma^{-1}$.

\begin{prop}
Let $(X,\sigma)$ be a two-sided subshift. If there is an uncountable set $Z\subset X$ such that $Z^2\subset \SProx(\sigma)\cap \Asy(\sigma^{-1})$, then $(X,\sigma)$ has a Cantor syndetically scrambled set.
\end{prop}
\begin{proof}
Fix $z\in Z$ and consider the Borel set $Z'=\{y\in X:(y,z)\in \SProx(\sigma)\cap \Asy(\sigma^{-1})\}$. Then $Z'$ is uncountable since $Z\subset Z'$, and so there is a Cantor set $K\subset Z'$. Now proceed as in the proof of Proposition \ref{onesideSy}.
\end{proof}

\subsection{From Cantor to Mycielski}

The next two Lemmas are related to our third question. These Lemmas will be used later as a unified tool in some proofs, making the exposition more compact. Also, since some of our results will remain true when `syndetically scrambled' is replaced with just `scrambled', we state this important fact as a general remark below.

\begin{rem}
While we do not state it explicitly, in many places it is possible to replace `syndetically scrambled' with just `scrambled' and the result still holds. Examples of such situation are Lemma \ref{cdense}, Lemma \ref{InvMyc}, Proposition \ref{joinCantScra} and Proposition \ref{t-to-one}. It is usually obtained by a straightforward modification of the proof, so the reader can easily adopt our proofs to this setting when necessary.
\end{rem}

\begin{lem}\label{cdense}
Let $(X,f)$ be a dynamical system and let $S\subset X$ be an uncountable syndetically scrambled set for $f$ with the following property: for any nonempty open set $U\subset X$ there is $n\in \mathbb N$ such that $f^n(S)\subset f^n(U)$.

Then, $f$ has a $c$-dense syndetically scrambled set $T$. If $S$ is also a Cantor set, then $T$ may be chosen to be a dense Mycielski set. Moreover, $T$ is $\eps$-scrambled if $S$ is $\eps$-scrambled.
\end{lem}
\begin{proof}
Let $\{V_j:j\in \mathbb N\}$ be a countable base for the topology of $X$ consisting of nonempty open sets of $X$, and for each $j\in \mathbb N$, let $U_j\subset X$ be a nonempty open set with $\overline{U_j}\subset V_j$. In the case that we assume only that $S$ is uncountable, choose pairwise disjoint uncountable sets $S_1,S_2,\ldots \subset S$. If we assume more, that is $S$ is a Cantor set, choose pairwise disjoint Cantor sets $S_1,S_2,\ldots \subset S$. Fix $j\in \mathbb N$, and let $n\in \mathbb N$ be such that $f^n(S_j)\subset f^n(U_j)$. Note that every power of $f$ is injective on $S_j$ since $S$ is scrambled.

First consider the case where $S$ is only uncountable. Then we can find an uncountable set $T_j\subset U_j$ and a bijection $\psi\colon S_j\to T_j$ such that $f^n(\psi(x))=f^n(x)$ for every $x\in S_j$. Put $T=\bigcup_{j=1}^\infty T_j$, which is clearly $c$-dense.

Next consider the case where $S$ is a Cantor set. Then $f^n(S_j)$ is a Cantor set since $f^n$ is injective and continuous on $S_j$. Let $X_j=\overline{U_j}\cap f^{-n}(f^n(S_j))$. Then $X_j$ is compact and $f^n\colon X_j\to f^n(S_j)$ is a continuous surjection. By \cite[Remark~4.3.6]{Sriv}, there is a Cantor set $T_j\subset X_j\subset V_j$ such that $f^n$ is injective on $T_j$ and $f^n(T_j)\subset f^n(S_j)$. Observe that $f^k$ is injective on $T_j$ for all $k\in \mathbb N$ (simply, use the injectivity of $f^k$ on $S_j$ for $k>n$). Put $T=\bigcup_{j=1}^\infty T_j$, and note that $T$ is a dense Mycielski set.

We now verify that $T$ is syndetically scrambled for $f$ in both the cases, as follows. Let $a,b\in T$ be distinct, and suppose $a\in T_i$, $b\in T_j$. Let $x\in S_i$, $y\in S_j$, $n,m\in \mathbb N$ be such that $f^n(x)=f^n(a)$ and $f^m(y)=f^m(b)$. Let if possible, $x=y$. Then $i=j$ and therefore $n=m$. This implies $f^n(a)=f^n(b)$, contradicting the injectivity of $f^n$ on $T_j$. Therefore we must have $x\ne y$, and hence $(x,y)\in \SProx(f)\setminus \Asy(f)$. Since
$f^{n+m+t}(a)=f^{n+m+t}(x)$ and $f^{n+m+t}(b)=f^{n+m+t}(y)$ for every $t\in \mathbb N$, we conclude that $(a,b)\in \SProx(f)\setminus \Asy(f)$. From this argument, the last assertion of the Lemma is also evident.
\end{proof}

The next Lemma is similar to the above, but the new ingredient is $f$-invariance, and for this we need deeper  topological arguments.

\begin{lem}\label{InvMyc}
Let $(X,f)$ be a dynamical system, and let $S\subset X$ be an $f$-invariant syndetically scrambled set for $f$ containing a Cantor set $K$. Suppose that for any nonempty open set $U\subset X$, there is $n\in \mathbb N$ such that $f^n(K)\subset f^n(U)$. Then, $f$ has a dense Mycielski $f$-invariant syndetically scrambled set $T$. Moreover, $T$ is $\eps$-scrambled if $S$ is $\eps$-scrambled.
\end{lem}
\begin{proof}
Note that all powers of $f$ are injective on $K$ since $K$ is contained in a scrambled set. Define an equivalence relation on $K$ by saying that for $x,y\in K$ we have $x\sim y$ iff $f^n(x)=f^m(y)$ for some $n,m\ge 0$. Then each equivalence class is countable since the powers of $f$ are injective on $K$. Also the equivalence relation is $F_\sigma$ in $K^2$ since $\{(x,y)\in K^2:x\sim y\}=\bigcup_{n=0}^\infty \bigcup_{m=0}^\infty \{(x,y)\in K^2:f^n(x)=f^m(y)\}$. Therefore, by Theorem 5.13.9 and Theorem 3.2.7 of \cite{Sriv}, there is a Cantor set $C\subset K$ such that $C$ intersects each equivalence class in at most one element. Choose pairwise disjoint Cantor sets $C_1,C_2,\ldots \subset C$. By the above construction we observe that $f^n(C_i)\cap f^m(C_j)\ne \emptyset$ iff $(i,n)=(j,m)$ for $i,j\in \mathbb N$ and $n,m\ge 0$.

Let $\{V_j:j\in \mathbb N\}$ be a countable base of nonempty open sets for $X$. If we fix $j\in \mathbb N$, then as in the proof of Lemma \ref{cdense} we can find $n\in \mathbb N$ and a Cantor set $T_j\subset V_j$ such that $f^n(T_j)\subset f^n(C_j)$ and $f^n$ is injective on $T_j$. As before, we also note that $f^k$ must be injective on $T_j$ for all $k\in \mathbb N$.

Put $T=\bigcup_{k=0}^\infty\bigcup_{j=1}^\infty f^k(T_j)$. Since $f^k$ is injective on $T_j$, the set $f^k(T_j)$ is Cantor, and therefore $T$ is a Mycielski set. Clearly $T$ is dense and $f$-invariant also.

Consider two distinct points $f^p(a),f^q(b)\in T$, where $a\in T_i$, $b\in T_j$ and $p,q\ge 0$. Then there are $x\in C_i, y\in C_j$ and $n,m\in \mathbb N$ such that $f^n(x)=f^n(a)$ and $f^m(y)=f^m(b)$. If $f^p(x)=f^q(y)$, then $(i,p)=(j,q)$ by our construction. Hence $a,b\in T_j$ and so $n=m$. Then
$f^{n+p}(a)=f^{n+p}(x)=f^{n+p}(y)=f^{n+p}(b)$ and consequently $a=b$ since $f^{n+p}$ is injective on $T_j$. This gives $f^p(a)=f^p(b)=f^q(b)$, a contradiction. Therefore we must have $f^p(x)\ne f^q(y)$. Then $(f^p(x),f^q(y))\in \SProx(f)\setminus \Asy(f)$ since $x$ and $y$ come from the $f$-invariant syndetically scrambled set $S$. For any $t\in \mathbb N$, we have $f^{n+m+t}(f^p(a))=f^{n+m+t}(f^p(x))$ and  $f^{n+m+t}(f^q(b))=f^{n+m+t}(f^q(y))$ and hence $(f^p(a),f^q(b))\in \SProx(f)\setminus \Asy(f)$.
\end{proof}

\subsection{Mycielski within Mycielski}

Sometimes it may be relatively easy to produce a dense Mycielski set $S=\bigcup_{k=1}^\infty S_k$ where $S_k$'s are Cantor sets, and $\eps>0$ such that any two elements of $S$ are syndetically proximal, and each $S_k$ is syndetically $\eps$-scrambled. The useful observation is that this data is sufficient to produce a dense Mycielski $\eps'$-scrambled set. To prove this, first we need an auxiliary Lemma.

\begin{lem}\label{lem:aux1}
Let $(X,f)$ be a dynamical system, $C\subset X$ be a perfect set, let $\eps>0$, $k\in \mathbb N$, and assume that for every open set $U$ with $C\cap U\neq \emptyset$, there is $n\ge k$ such that $\diam[f^n(C\cap U)]>2\eps$. Let $t\ge 2$ and let $U_1,\ldots ,U_t\subset X$ be open sets with $C\cap U_i\ne \emptyset$ for $1\le i\le t$. Then:
\begin{enumerate}[(i)]
\item\label{313:i} There are $x_i\in C\cap U_i$ for $1\le i\le t$ with the property that for every $1\le i<j\le t$ there is $n=n(i,j)\ge k$ with $d(f^n(x_i),f^n(x_j))>\eps$.

\item\label{313:ii} There are open sets $V_1,\ldots ,V_t \subset X$ such that $\overline{V_i}\subset U_i$, $V_i\cap C\ne \emptyset$ and such that for every $1\le i<j\le t$, there is $n=n(i,j)\ge k$ with $dist (f^n(\overline{V_i}), f^n(\overline{V_j}))>\eps$.
\end{enumerate}
\end{lem}
\begin{proof}
First we prove \eqref{313:i}. Let $x_1\in C\cap U_1$. By hypothesis, there are $a,b\in C\cap U_2$ and $n\ge k$ such that $d(f^n(a),f^n(b))>2\eps$. Then either $d(f^n(x_1),f^n(a))>\eps$ or $d(f^n(x_1),f^n(b))>\eps$. So there is $x_2\in \{a,b\}\subset C\cap U_2$ such that $d(f^n(x_1),f^n(x_2))>\eps$. Now let $2\le s\le t$, and assume we have chosen $x_i\in C\cap U_i$ for $1\le i<s$ and $n(i,j)\ge k$ for $1\le i<j<s$ such that $d(f^{n(i,j)}(x_i),f^{n(i,j)}(x_j))>\eps$. If we consider $x_1$ and $U_s$, then as in the first step we can find $y\in C\cap U_s$ and $n(1,s)\ge k$ such that $d(f^{n(1,s)}(x_1),f^{n(1,s)}(y))>\eps$. Then there is an open set $W_{s,1}$ such that $y\in W_{s,1}\subset U_s$ and $\dist(f^{n(1,s)}(x_1),f^{n(1,s)}(\overline{W_{s,1}}))>\eps$. Next we repeat the same argument by considering $x_2$ and $W_{s,1}$, etc. Thus we can construct a finite decreasing sequence $W_{s,1}\supset W_{s,2}\supset \cdots \supset W_{s,s-1}$ of open sets intersecting $C$ and $n(i,s)\ge k$ for $1\le i<s$ such that $\dist(f^{n(i,s)}(x_i),f^{n(i,s)}(\overline{W_{s,i}}))>\eps$. Let $x_s$ be any point from $C\cap W_{s,s-1}$. The proof of \eqref{313:i} is completed.

To prove \eqref{313:ii} it is enough to choose points $x_1,\ldots,x_t$ as in part \eqref{313:i}, and then choose sufficiently small open neighborhoods $V_i$ of $x_i$ for $1\le i\le t$.
\end{proof}

\begin{prop}\label{joinCantScra}
Let $(X,f)$ be a dynamical system. Suppose there are $\eps>0$ and a dense Mycielski set $S\subset X$ with the following properties:

\noindent (i) $S^2\subset \SProx(f)$.

\noindent (ii) For any nonempty open set $W\subset X$, there is a Cantor set $C\subset S\cap W$ such that $C$ is syndetically $3\eps$-scrambled for $f$.

Then $f$ has a dense Mycielski syndetically $\eps$-scrambled set $T$ with  $T\subset S$.
\end{prop}
\begin{proof}
Fix a countable base $\{W_p:p\in \mathbb N\}$ consisting of nonempty open sets for the topology of $X$. Let $C_1\subset S\cap W_1$ be a syndetically $3\eps$-scrambled set for $f$. Replacing $C_1$ with a smaller Cantor subset, we may assume that $C_1$ is nowhere dense in $X$ (e.g. $C_1$ is homeomorphic with $C_1\times C_1$, where any vertical section is nowhere dense). If $p>1$, then at the $p$-th step, note that $W_p':=W_p\setminus \bigcup_{q=1}^{p-1}C_q$ is a nonempty open set since $C_q$'s are nowhere dense closed sets. So there is a Cantor set $C_p\subset S\cap W_p'$ such that $C_p$ is syndetically $3\eps$-scrambled for $f$. Again we may assume that $C_p$ is nowhere dense in $X$. By this construction, $C_p$'s are pairwise disjoint Cantor sets with $C_p\subset S\cap W_p$, and $\bigcup_{p=1}^\infty C_p\subset S$. Moreover, each $C_p$ is a syndetically $3\eps$-scrambled set for $f$. By the last assertion note that if $U\subset X$ is an open set with $C_p\cap U\ne \emptyset$, then $\diam[f^n(C_p\cap U)]>2\eps$ for infinitely many $n\in \mathbb N$.

Let $V(1,1,0),V(1,1,1)\subset W_1$ be open sets with disjoint closures such that $C_1\cap V(1,1,a)\ne \emptyset$ and $\diam [V(1,1,a)] < 1$ for $a=0,1$. Replacing these two open sets by smaller ones with the help of Lemma~\ref{lem:aux1}, we may also assume there is $n\ge 1$ such that $\dist(f^n(\overline{V(1,1,0)}),f^n(\overline{V(1,1,1)}))>\eps$.

By induction on $m\in \mathbb N$, we are going to choose open sets $V(m,p,w)\subset W_p$ for all $m\geq 1$, $1\le p\le m$ and $w\in \{0,1\}^{m+1-p}$ such that:

\begin{enumerate}[(i)]
\item $\overline{V(m,p,w)}\cap \overline{V(m,q,u)}=\emptyset$ if $(p,w)\ne (q,u)$.

\item $C_p\cap V(m,p,w)\ne \emptyset$.

\item $\diam [V(m,p,w)]<1/m$.

\item $V(m,p,w)\subset V(m-1,p,v)$ if $w=va$ for some $a\in \{0,1\}$.

\item If $(p,w)\neq (q,u)$, then there is $n\ge m$ with the property that
$$\dist(f^n(\overline{V(m,p,w)}),f^n(\overline{V(m,q,u)}))>\eps.$$
\end{enumerate}

Conditions (i)-(iv) are very easy to satisfy, and condition (v) can be obtained by an application of Lemma~\ref{lem:aux1}.

For each fixed integer $p\in \mathbb N$, put $K_p=\bigcap_{m=p}^\infty \bigcup_{w\in \{0,1\}^{m+1-p}} \overline{V(m,p,w)}$. By the construction we easily get that $K_p$ is a Cantor set and by properties (ii), (iii) and the compactness of $C_p$, we also have $K_p\subset C_p\subset S\cap W_p$. Therefore $T:=\bigcup_{p=1}^\infty K_p$ is a dense Mycielski subset of $X$ with $T\subset S$. Clearly $T^2\subset \SProx(f)$ since $S^2\subset \SProx(f)$.

Now consider two distinct points $y_1,y_2\in T$ and fix any $k\in \mathbb N$. Let $m\ge k$ be such that $1/m<d(y_1,y_2)$ and choose $p,q\in \{1,\ldots,m\}$, and words $w\in \{0,1\}^{m+1-p}$, $u\in \{0,1\}^{m+1-q}$ such that $y_1\in \overline{V(m,p,w)}$ and $y_2\in \overline{V(m,q,u)}$. The choice of $m$ guarantees by property (iii) that $(p,w)\neq (q,u)$. Then by property (v), there is $n\ge m> k$ such that $d(f^n(y_1),f^n(y_2))>\eps$. We conclude that $\limsup_{n\to \infty} d(f^n(y_1),f^n(y_2))\ge \eps$, since $k\in \mathbb N$ was arbitrary. So $T$ is syndetically $\eps$-scrambled for $f$.
\end{proof}

\subsection{Through factor maps}

The aim of this subsection is to present a few conditions that help to transfer (syndetically) scrambled sets to factors. Factor maps preserve (syndetically) proximal pairs, but a non-asymptotic pair may be taken to an asymptotic pair by the factor map. Therefore, to ensure the existence of a (syndetically) scrambled set with possibly additional properties in the factor system, extra assumptions are required. Here we offer a few results of this kind and later, in coming sections we will present their possible applications.

Recall that a factor map $\pi\colon (X,f)\to (Y,g)$ is proximal, if $\pi(x)=\pi(y)$ implies that $(x,y)\in \Prox (f)$.
Denote by $R_\pi=\set{(x,y) \in X\times X : \pi(x)=\pi(y)}$ and note that if $\pi$ is proximal then $R_\pi\subset \Prox(f)$.
We start this subsection with the following useful fact.

\begin{lem}
If $(Y,g)$ is distal and $\pi\colon (X,f)\to (Y,g)$ is proximal then $R_\pi=\Prox(f)=\SProx(f)$.
\end{lem}
\begin{proof}
Note that if $(x,y)\in \Prox(f)$ then $(\pi(x),\pi(y))\in \Prox(g)$, hence $R_\pi=\Prox(f)$.
Observe that $R_\pi$ is closed and invariant for $f\times f$. If we fix any open set $U\supset \Delta$
and any $(x,y)\in R_\pi$ then $N((x,y),U)$ is syndetic, as otherwise there is a minimal set $M\subset R_\pi \setminus U$
which is impossible because $M\cap \Prox(f)=\emptyset$. This proves that $R_\pi=\SProx(f)$ and the proof is complete.
\end{proof}

\begin{cor}\label{prop:syprox}
Let $\pi\colon (X,f) \ra (Y,g)$ be a factor map between dynamical systems $(X,f)$, $(Y,g)$.
Assume additionally that $g$ is distal, minimal and $\pi^{-1}(z)$ is singleton for some $z\in Y$. Then $\Prox(f)=\SProx(f)$.
\end{cor}
\begin{proof}
Fix any $x_1,x_2\in X$ such that $\pi(x_1)=\pi(x_2)=y$. There exists an increasing sequence $n_k$ such that $\lim_{k\to \infty}g^{n_k}(y)=z$.
Passing to a subsequence if necessary, we may assume that there are $q_1,q_2\in X$ such that $\lim_{k\to \infty}f^{n_k}(x_1)=q_1$ and $\lim_{k\to \infty}f^{n_k}(x_1)=q_2$ and hence $\pi(q_1)=z=\pi(q_2)$. It implies that $q_1=q_2$ and so $(x_1,x_2)\in \Prox(f)$ which ends the proof.
\end{proof}

Recall that $f\colon X \ra X$ is \emph{almost 1-1 extension} of $g\colon Y\ra Y$ if there is a factor map $\pi \colon X \ra Y$
such that the set $Y'=\bigcap_{n=1}^\infty \set{y \; : \; \diam [\pi^{-1}(y)]<1/n}$ is residual in $Y$.

\begin{cor}\label{cor:syprox}
If $f$ is an almost 1-1 extension of a distal and minimal system then $\Prox(f)=\SProx(f)$.
\end{cor}

While attempting to transfer a scrambled set to a factor system, we may ask if does it help to assume that our factor map is finite-to-one or countable-to-one? The answer is in the affirmative, however some additional assumptions are needed. This is what we discuss next.

\begin{prop}\label{t-to-one}
Let $(X,f)$, $(Y,g)$ be dynamical systems, let $t\in \mathbb N$, and let $\pi\colon (X,f)\to (Y,g)$ be a factor map such that $\#\pi^{-1}(y)\le t$ for every $y\in Y$. Let $S\subset X$ be an uncountable syndetically scrambled set for $f$ and suppose that for any collection $x_1,\ldots ,x_{t+1}$ of $t+1$ distinct points in $S$, there are $\eps>0$ and an infinite set $M\subset \mathbb N$ such that $d(f^n(x_i),f^n(x_j))\ge \eps$ for every $n\in M$ and every $1\le i<j\le t+1$. Then, $\pi(S)$ contains an uncountable syndetically scrambled set for $g$.
If $S$ is a Cantor set, then $\pi(S)$ contains a Cantor syndetically scrambled set for $g$.
\end{prop}

\begin{proof}
Clearly any two elements of $\pi(S)$ are syndetically proximal for $g$. Therefore, we only have to take care of the asymptotic relation. Keep in mind that the asymptotic relation is an equivalence relation.

First we clam that if $y_1,\ldots,y_{t+1}\in \pi(S)$ are such that $(y_i,y_j)\in \Asy(g)$ for every $i,j$, then there exist $i\ne j$ such that $y_i=y_j$.

Suppose that the claim is false, and let $x_i\in S\cap \pi^{-1}(y_i)$. Then $x_i$'s are distinct. Choose $\eps>0$ and an infinite set $M\subset \mathbb N$ as per the hypothesis for $x_1,\ldots,x_{t+1}$. By compactness, we can find a sequence $(n_k)$ in $M$ and $b\in Y$ such that $(g^{n_k}(y_i))\to b$ for $1\le i\le t+1$. By passing onto a subsequence, we can find points $a_i\in X$ such that $(f^{n_k}(x_i))\to a_i$ for $1\le i\le t+1$. Clearly $d(a_i,a_j)\ge \eps$ for $i\ne j$, and thus $a_i$'s are distinct. Since we must also have $\pi(a_i)=b$, we obtain $\#\pi^{-1}(b)\ge t+1$, a contradiction. This ends the proof of the claim.

Now, by the above claim, for each $y\in \pi(S)$, the intersection of $\pi(S)$ with the asymptotic cell of $y$ can contain at most $t$ elements. Taking exactly one point from each asymptotic equivalence class that intersects $\pi(S)$, we get the required set.

Finally, let us consider the case when $S$ is a Cantor set. Then $\pi(S)$ is also a Cantor set and by previous arguments, equivalence classes of $\Asy(g)$ over $\pi(S)\times \pi(S)$ are finite. Note that $\Asy(g)$ is a Borel set, since it is $F_{\sigma\delta}$. Thus $\Asy(g)$ is a Borel equivalence relation on $\pi(S)\times \pi(S)$ with uncountably many equivalence classes, and thus by \cite[Theorem 5.13.4]{Sriv} we can find a Cantor set $D\subset \pi(S)$ such that if $x,y\in D$ and $(x,y)\in \Asy(g)$ then $x=y$. The proof is finished.
\end{proof}

Next we consider the case of countable-to-one factor maps, that is factor maps $\pi\colon (X,f)\to (Y,g)$ such that $\pi^{-1}(\set{y})$ is at most countable set for every $y\in Y$.

\begin{prop}\label{countable-to-one}
Let $(X,f)$, $(Y,g)$ be dynamical systems, let $\pi\colon (X,f)\to (Y,g)$ be a countable-to-one factor map, and assume that $z_1,z_2\in X$ are so that $\pi(z_1)\ne \pi(z_2)$. Let $\eps=d(\pi(z_1),\pi(z_2))/3$, and let $\delta>0$ be such that $d(a,b)<\delta$ implies $d(\pi(a),\pi(b))<\eps$ for $a,b\in X$. Suppose that  $S\subset X$ is a dense Mycielski  syndetically scrambled set for $f$ with the following property: for any two distinct points $x_1,x_2\in S$, there is an infinite set $M\subset \mathbb N$ such that $d(z_i,f^n(x_i))<\delta$ for $i=1,2$ and every $n\in M$. Then $g$ has a dense Mycielski syndetically $\eps$-scrambled set $T\subset \pi(S)$. Moreover, if $S$ is $f^k$-invariant for some $k\in \mathbb N$, then $T$ may be chosen to be $g^k$-invariant.
\end{prop}
\begin{proof}
First we verify that $\pi(S)$ is syndetically $\eps$-scrambled for $g$. Clearly $\pi(S)^2\subset \SProx(g)$. Now consider two distinct points $y_1,y_2\in \pi(S)$, and let $x_1,x_2\in S$ be so that $\pi(x_i)=y_i$ for $i=1,2$. By assumption, there is an infinite set $M$ such that $d(z_i,f^n(x_i))<\delta$ for $i=1,2$ and all $n\in M$. Then $d(\pi(z_i),g^n(y_i))=d(\pi(z_i),\pi(f^n(x_i)))<\eps$ for $i=1,2$ and hence $d(g^n(y_1),g^n(y_2))>\eps$ for every $n\in M$ by triangle inequality and the choice of $\eps$.

This shows that $\limsup_{n\to \infty}d(g^n(y_1),g^n(y_2))\ge \eps$, and thus $\pi(S)$ is syndetically $\eps$-scrambled for $g$.

Let $\{W_i:i\in \mathbb N\}$ be a countable base of the topology of $Y$ consisting of nonempty open sets. Fix $i$ and consider $W_i$. Since $S$ is a dense Mycielski set in $X$, we have that $S\cap \pi^{-1}(W_i)$ contains a Cantor set, say $K_i$. Since $\pi$ is countable-to-one, $\pi(K_i)\subset \pi(S)\cap W_i$ is an uncountable perfect set and hence there is a Cantor set $C_i\subset \pi(K_i) \subset \pi(S)\cap W_i$. Put $T=\bigcup_{i=1}^\infty C_i$. Then $T$ is Mycielski, $T$ is dense in $Y$ since $T\cap W_i\ne \emptyset$ for every $i\in \mathbb N$, and $T$ is syndetically $\eps$-scrambled for $g$ since $T\subset \pi(S)$.

Next, assume in addition that $S$ is $f^k$-invariant. Then $\pi(S)$ is $g^k$-invariant. Let $T=\bigcup_{i=1}^\infty C_i$ be as constructed above and let $T'=\bigcup_{n=0}^\infty \bigcup_{i=1}^\infty g^{kn}(C_i)$. Clearly $T'$ is $g^k$-invariant, and $T'$ is syndetically $\eps$-scrambled for $g$ since $T'\subset \pi(S)$. Note that $g^{kn}$ is injective on $\pi(S)$ for any $n\in \mathbb N$ since $\pi(S)$ is syndetically scrambled for $g^{kn}$ also. In particular, $g^{kn}$ is injective on $C_i$ and hence $g^{kn}(C_i)$ is a Cantor set. But $T\subset T'$, therefore $T'$ is a dense Mycielski set.
\end{proof}

\begin{rem}
In the above hypothesis, the points $z_1,z_2$ are independent of $x_1,x_2$. But it does not matter if they depend on $x_1,x_2$ also, as long as there is a uniform $\eps$ for any pair. That is, the hypothesis could also be: ``let $\eps>0$, $\delta>0$ be such that $d(a,b)<\delta$ implies $d(\pi(a),\pi(b))<\eps$ for $a,b\in X$; and suppose that  $S\subset X$ is a dense Mycielski  syndetically scrambled set for $f$ with the following property: for any two distinct points $x_1,x_2\in S$, there are $z_1,z_2\in X$ and an infinite set $M\subset \mathbb N$ such that $d(\pi(z_1),\pi(z_2))\ge 3\eps$ and $d(z_i,f^n(x_i))<\delta$ for $i=1,2$ and every $n\in M$".
\end{rem}

As we have seen, if $\pi\colon (X,f)\to (Y,g)$ is finite-to-one or countable-to-one, syndetically scrambled set can be transferred from $X$ to $Y$. Unfortunately, this does not work in the opposite direction.

\begin{exmp}
There is a subshift $X\subset \Sigma_2^+ $ and factor map $\pi\colon (X,\sigma)\to (Y,g)$, such that $\pi$ is one-to-one for all but one point and the class $\SProx(\sigma|_X)(x)$ is at most countable for any $x\in X$ (in particular all syndetically scrambled sets are at most countable) while $Y$ contains a Cantor syndetically scrambled set for $\sigma$.
\end{exmp}
\begin{proof}
Let $D\subset \Sigma_2^+$ be a Cantor scrambled set for $\sigma$ and for any $x\in D$ define the point $z^x=w(0)w(1)\ldots$ where
$w(n)=0^{n+1}$ if $x_n=0$ and $1^{n+1}$ if $x_n=1$. Let $X$ be the smallest (in the sense of inclusion) subshift containing all the points $z^x$, $x\in  D$.
Note that if $x,y$ are distinct then for every $n,m\geq 0$ there exists a thick set $I$ such that $z^x_i\neq z^y_i$ for every $i\in I$.
Additionally, $X$ consists exactly of orbits of points $z^x$ together with their limit points, that is sequences $1^k 0^\infty$, $0^k 1^\infty$ where $k\geq 0$.
Then syndetically proximal class of any $z\in X$ is at most countable. Additionally, for every $\eps>0$ the set
$$
\set{n: d(\sigma(z^x),0^\infty)<\eps}\cup \set{n: d(\sigma(z^x),1^\infty)<\eps}
$$
is syndetic. Therefore, if $Y$ is obtained from $X$ by identifying points $0^\infty$ and $1^\infty$, say by a factor map $\pi$, then the set $S=\set{\pi(z^x): x\in D}$ becomes a Cantor syndetically scrambled set, which ends the proof.
\end{proof}

However, we wish to note one special situation where we can transfer a syndetically scrambled set upwards.
This can be successfully applied (with some extra work) in some situations arising from applications (e.g. see \cite{WZ}).

\begin{prop}
Let $\pi:(X,f)\to (Y,g)$ be a factor map between dynamical systems and let $T\subset Y$ be a syndetically scrambled set containing a fixed point $b$ for $g$. If $\pi^{-1}(b)$ is a singleton, then $\pi^{-1}(T)$ is a syndetically scrambled set for $f$. Moreover, if $T$ is $\eps$-scrambled and if $\delta>0$ is chosen for this $\eps$ using the uniform continuity of $\pi$, then $\pi^{-1}(T)$ is $\delta$-scrambled.
\end{prop}
\begin{proof}
Let $S=\pi^{-1}(T)$ and $a\in X$ be the unique point with $\pi(a)=b$. Clearly, $f(a)=a$. Let $x\in S$ and consider a minimal point $c\in \overline{\orbp(x,f)}$. Then $\pi(c)$ is a minimal point in $\overline{\orbp(\pi(x),g)}$. But $\pi(x)\in T$ and therefore $b$ is the unique minimal point in $\overline{\orbp(\pi(x),g)}$ (see Proposition \ref{uniMin} later), and hence $\pi(c)=b$, or $c=a$. This shows that $a$ is the unique minimal point in $\overline{\orbp(x,f)}$ for any $x\in S$. Consequently $(x,a)\in SyProx(f)$ for every $x\in S$, and so $S^2\subset SyProx(f)$. If $x_1,x_2\in S$ are distinct, then $(\pi(x_1),\pi(x_2))\notin Asy(g)$ and therefore $(x_1,x_2)\notin Asy(f)$. Thus $S$ is syndetically scrambled for $f$. The last assertion is easy to verify.
\end{proof}

\section{Transitive subshifts}\label{sec:transSD}

\subsection{The full shift}

We start our analysis of symbolic systems with the simplest case of the full shift over a finite alphabet.
In this case our conclusions will be the strongest possible. Since in many cases the full shift can be embedded into
dynamics, results of this section can be of particular interest.

\begin{thm}\label{thm:full_shift}
If $m>1$ then both $(\Sigma_m^+,\sigma)$ and $(\Sigma_m,\sigma)$ have dense Mycielski $\sigma$-invariant syndetically scrambled sets.
\end{thm}
\begin{proof}
First consider the case of $\Sigma_m^+$. Let $C\subset \Sigma_2^+$ be a Cantor scrambled set for $\sigma$ and let $g\colon \Sigma_2^+\to \Sigma_2^+\subset \Sigma_m^+$ be defined as $g(x)=y$, where

$y_n=
\begin{cases}
x_k,\text{ if }n=2^k\\
0, \text{ otherwise.}
\end{cases}$

Then $g$ is injective and continuous, and thus $g(C)$ is a Cantor set. Therefore $S:=\bigcup_{n=0}^\infty \sigma^n(g(C))$ is a $\sigma$-invariant Mycielski set in $\Sigma_m^+$. It is not difficult to verify (see proof of Theorem 7 in \cite{tksm}) that $S$ is syndetically scrambled for $\sigma$. Also note that the one-sided full shift is an exact map. So by Lemma \ref{InvMyc}, there is a dense Mycielski $\sigma$-invariant syndetically scrambled set for $\sigma$.

Next consider the case of two sided subshift. Let $g(C)\subset \Sigma_m$ be the Cantor scrambled set constructed above and let $\set{C_j}_{j=0}^\infty\subset g(C)$ be disjoint Cantor sets. Let $\set{w_j}_{j=0}^\infty$ be the sequence of all words that can be constructed over $m$ letters alphabet (i.e. all possible subwords of elements of $\Sigma_m^+$). Now we extend every element $z$ in $C_j$ to a bi-infinite sequence $z'$ by putting $z'_{(-\infty,0)}=0^\infty w_j$. Denote by $C_j'\subset \Sigma_m$ the set of all such extended sequences, where $j=0,1,2,\ldots$.
Let $D_1=C_1'$ and put $n(1)=1$. Assume now that sets $D_1,\ldots, D_k$ as well as numbers $n(1)<n(2)<\ldots <n(k)$ are defined, and let
$n(k+1)$ be the smallest integer (if it exists, or infinity otherwise) such that $w_{n(k+1)}$ does not appear as a subword of any bi-infinite sequence $z\in \bigcup_{i=1}^k D_k$,
or equivalently the trajectory of $z$ never intersects the cylinder set $C[w_{n(k+1)}]$. If $n(k+1)<\infty$ then we put $D_{k+1}=C_{n(k+1)}'$, otherwise we end the construction. That way a finite or infinite sequence of sets $D_j$ is constructed, and let $I$ be the set of such indices, i.e. $j\in I$ iff $D_j$ was constructed. By the construction we can never have $\sigma^n(x)=\sigma^m(y)$ for any $n,m\in \Z$ provided that $x\in D_i$, $y\in D_j$ and $i\neq j$.

Note that the set $S:=\bigcup_{n=-\infty}^\infty \bigcup_{j\in I} \sigma^n(D_j)$ is syndetically scrambled Mycielski set and $\sigma(S)=S$ (exactly the same arguments as for the case of $\Sigma_m^+$ can be applied). But it is also dense, since $S$ intersects every cylinder set in $\Sigma_m$.
\end{proof}

\begin{rem}
In the above construction, the syndetically scrambled set contains the fixed point $0^\infty$ and hence the syndetically scrambled set is disjoint with $\Tran(\sigma)$. We can also obtain syndetically scrambled sets contained in $\Tran(\sigma)$. For this, observe that the hypothesis of the following result is applicable to the full shift. Simply, $\Sigma_m$ (as well as $\Sigma_m^+$) has a natural group structure with respect to coordinatewise addition modulo $m$, where $0^\infty$ is the identity element, and then the full shift is a group automorphism.
\end{rem}

\begin{thm}
Let $(X,*)$ be an infinite compact metrizable group (need not be Abelian) with identity element $e$ and let $f\colon X\to X$ be a continuous homomorphism such that $f$ is transitive and $\overline{\M(f)}=X$. Let $S\subset X$ be a syndetically scrambled set for $f$ with $e\in S$, and let $x\in \Tran(f)$. If $T=S*x$, then we have the following:

\begin{enumerate}[(i)]
\item $T$ is homeomorphic to $S$, and $T$ is a syndetically scrambled set for $f$. If $S$ is $\eps$-scrambled, so is $T$. (In fact, these are true for any $x\in X$.)

\item $T\subset \Tran(f)$; and in particular, $\overline{\orbp(y,f)}$ contains infinitely many distinct minimal sets for any $y\in T$.

\item $\SProx(f)(y)\cap \M(f)=\emptyset$ for every $y\in T$.
\end{enumerate}
\end{thm}

\begin{proof}
Any compact metrizable group admits a two-sided invariant metric (see Theorem 8.6 of \cite{hero}), that is $d(a,b)=d(a*z,b*z)=d(z*a,z*b)$ for any $a,b,z\in X$. Let $d$ be one such metric on $X$.

\noindent (i) In any topological group, the right multiplication by a fixed element is a homeomorphism, and therefore $T$ is homeomorphic to $S$. For the other statements, note that if $a,b\in S$, then $d(f^n(a*x),f^n(b*x))=d(f^n(a)*f^n(x),f^n(b)*f^n(x))=d(f^n(a),f^n(b))$ by the invariant nature of the metric $d$.

\noindent (ii) Let $y=a*x\in T$, where $a\in S$, let $z\in M(f)$, and let $\eps>0$. We have that $\{n\in \mathbb N:d(e,f^n(a))<\eps/2\}$ is thickly syndetic and $\{n\in \mathbb N:d(z,f^n(x))<\eps/2\}$ is piecewise syndetic. Let $n\in \mathbb N$ be in the intersection of the above two sets. Then, $d(z,f^n(y))=d(z,f^n(a*x))=d(z,f^n(a)*f^n(x))\le d(z,f^n(x))+d(f^n(x),f^n(a)*f^n(x))=d(z,f^n(x))+d(e,f^n(a))<\eps$. This argument shows that $M(f)\subset \overline{\orbp(y,f)}$, and therefore $y\in \Tran(f)$ since $\overline{M(f)}=X$ by assumption. The second assertion in (ii) can be deduced as follows: since $f$ is a transitive non-minimal map, any minimal set of $f$ must be nowhere dense; so there must be infinitely many minimal sets since $\overline{\M(f)}=X$.

Finally (iii) follows from (ii) and Proposition \ref{uniMin} proved later in this article (simply $\SProx(f)(y)\cap M(f)=\emptyset$ for every $y\in T$).
\end{proof}

\subsection{Subshifts of finite type}

It is not difficult to show that if $(X,\sigma)$ is an infinite transitive subshift of finite type (SFT), then there are $m\ge 2$ and $k\in \mathbb N$ such that a copy of the $k$-th power of the full shift on $m$ symbols is embedded in $(X,\sigma)$ (see the proof of Theorem \ref{SFTperiod} below). Therefore we easily obtain the following.

\begin{cor}\label{cor:inf_trans_SFT}
If $(X,\sigma)$ is an infinite transitive SFT, then there are $k\in \mathbb N$ and a $\sigma^k$-invariant Mycielski set  $S\subset X$ such that $S$ is syndetically scrambled for $\sigma$.
\end{cor}

Then the following question is natural to ask: is it possible to make $S$ dense in $X$ and $\sigma$-invariant at the same time?
First of all we must assume that $(X,\sigma)$ is mixing as otherwise we have a regular periodic decomposition of $X$ by disjoint sets, which is an obstruction for
density of scrambled set.

If we consider mixing SFT with a fixed point (which must be contained in the system, if $S$ is invariant scrambled set), then the above question has the affirmative answer. Before proving it, we need an auxiliary lemma.

We say that a subshift $X$ over alphabet $\Al$ is a \emph{vertex shift} if there is a directed graph with set of vertices equal to $\Al$ such that
the set of infinite paths (resp. bi-infinite paths for two-sided shift case) is equal to $X$. Obviously, every vertex shift is SFT given by a set of forbidden words $F\subset \Al^2$, and conversely, it can be proved (e.g. see \cite{LM}) that every SFT is conjugated to a vertex shift (possibly, over a larger alphabet).
First we need the following lemma (here $\alp(w)=\set{w_i : i=0,\ldots, |w|-1}$).
\begin{lem}\label{lem:SFTinvS}
Let $X$ be an infinite transitive vertex shift and let $x\in X$ be a periodic point. There exist words $u,w$ so that $u_0=w_0$, $\sigma^p(x)=w^\infty$ for some $p\in \N$, $u^\infty \in X$ and $\alp(u)\neq \alp(w)$.
\end{lem}
\begin{proof}
Let $G$ be the graph defining $X$ and let $A$ be the set of vertices, that is the alphabet of $X$.
Fix $w$ such that $x=w^\infty$. If $\alp(w)\neq A$, take $u$ to be a word representing a cycle passing through $w_0$ and some $b\in A\setminus \alp(w)$. If $\alp(w)=A$, let $t$ be the minimal length of a cycle in the graph $G$. If $t=|A|$, then $G$ reduces to a cycle of length $|A|$, forcing $X$ to be a finite set, which is a contradiction. Hence $t<|A|$, and let $u$ be a word representing a cycle of length $t$. Since $u_0\in A=\alp(w)$, replacing $x$ with $\sigma^p(x)$ for some $p\ge 0$ and $w$ with its cyclic permutation, we can have $\sigma^p(x)=w^\infty$ and $u_0=w_0$.
\end{proof}

\begin{thm}\label{SFTperiod}
Let $(X,\sigma)$ be an infinite mixing SFT and let $x\in X$ be a periodic point with period $p$. Then there is a dense Mycielski syndetically scrambled set $S\subset X$ invariant for $\sigma^p$ and such that $x\in S$.
\end{thm}

\begin{proof}
We prove the case of one-sided shift. The case of two-sided shift can be proved similarly (see the proof of Theorem~\ref{thm:full_shift}). Without loss of generality we may assume that $X$ is a vertex shift, in particular all forbidden words are of length 2.
Let $v,w$ be words provided by Lemma~\ref{lem:SFTinvS} for $x$, in particular $\alp(v)\neq \alp(w)$.
If $S$ is $\sigma^p$-invariant syndetically scrambled set, then $f$ is injective on $S$ and so $f^n(S)$ is also $\sigma^p$-invariant syndetically scrambled set. Thus without loss of generality we may assume that $x=w^\infty$. Then clearly we may assume that $|w|=p$ and denote $|v|=q$. Put $a=w_0$.

Consider the words $B=w^q$ and $C=v^p$. Then $|B|=|C|=pq$, $B\ne C$, and any sequence obtained by the concatenation of $B$'s and $C$'s belongs to $X$. Let $K'\subset \Sigma_2^+$ be a Mycielski $\sigma$-invariant syndetically  scrambled set for $\sigma$ such that $0^\infty \in K'$ and let $K\subset X$ be the corresponding set obtained by replacing $0$ by $B$ and $1$ by $C$ in each element of $K'$. In particular, note that $x=w^\infty=B^\infty\in K$. Choose pairwise disjoint Cantor sets $K_0,K_1,K_2,\ldots \subset K$ in such a way that $x\in K_0$. Let $u(0),u(1),u(2),\ldots $ be a listing of all words in $L(X)$. We may assume that $u(0)=a$. Since $(X,\sigma)$ is a mixing SFT, for each $i\in \mathbb N$, there is a word $u(i)'$ such that $|u(i)u(i)'|$ is a multiple of $pq$ and $u(i)u(i)'a\in L(X)$. Let $S_0=K_0$ (and note that $x\in S_0$). For $i\in \mathbb N$, let $S_i\subset X$ be the Cantor set obtained by replacing the initial word of length $|u(i)u(i)'|$ of each $y\in K_i$ by $u(i)u(i)'$.
Elements obtained after this modification remain in $X$ since both the words $B$ and $C$ start with the letter $a$ and $X$ is a vertex shift. Finally put $D=\bigcup_{i=0}^\infty S_i$. It may be verified that $D$ is a dense Mycielski syndetically scrambled set.
Furthermore, since $K$ was $\sigma^{pq}$-invariant, for any $y,z\in D$ and any $i,j\in \N$ (if $y=z$ then we assume $i\neq j$) the pair $(\sigma^{ipq}(y),\sigma^{jpq}(z))$ is syndetically scrambled. Then we can find arbitrarily large $s$ such that
$$
y_{[(i+s)pq,(i+s+1)pq)}=B \quad\text{ and }\quad z_{[(j+s)pq,(j+s+1)pq)}=C
$$
or vice-versa. We will consider only the case when there is $b\in \alp(B)\setminus \alp(C)$, symmetric case $b\in \alp(C)\setminus \alp(B)$ can be proved following similar lines. Denote by $r,t$, respectively, the first and the last occurrence of symbol $b$ in $B$. Note that $r<p$ and $pq-t<p$ since $B=w^q$ and $|w|=p$. Additionally observe that $z_{l+(j+s)pq}\neq b$ for any $l\in (t-pq,r+pq)$.
This shows that for any $m\in [0,pq]$ pairs $(\sigma^{ipq+m}(y),\sigma^{jpq}(z))$ and $(\sigma^{ipq}(y),\sigma^{jpq+m}(z))$ are syndetically scrambled.
Note that the pair $(x,y)$ is syndetically proximal for every $y\in D$, and so  $(\sigma^{ip}(y),x)$ is also syndetically proximal for any $i\geq 0$.
Therefore the set $S=\bigcup_{i=0}^\infty \sigma^{ip}(D)$ is $\sigma^p$-invariant syndetically scrambled set.
Of course it is also dense Mycielski scrambled set, since $\sigma$ is injective on $D$.
\end{proof}

\subsection{Synchronizing subshifts}

\begin{prop}
Let $(X,\sigma)$ be an infinite one-sided or two-sided synchronizing subshift. Then there is a Mycielski, syndetically scrambled set which is $\sigma^k$-invariant for some integer $k\in \N$. \end{prop}
\begin{proof}
Let $s\in L(X)$ be a synchronizing word. Since $X$ is infinite, there are $u,u'\in L(X)$ such that $|u|=|u'|$, $u\ne u'$, and $su,su'\in L(X)$. By transitivity, find $v,v'\in L(X)$ such that $suvs, su'v's\in L(X)$ (here $v$ and $v'$ may have different lengths). Define two words $B,C$ as $B=suvsu'v'$, $C=su'v'suv$. Then $B,C\in L(X)$, $|B|=|C|$, $B\ne C$, and any concatenation of $B$'s and $C$'s belongs to $L(X)$.

This shows that there is a set $\Lambda\subset X$ such that $\sigma^k(\Lambda)=\Lambda$ and systems $(\Lambda,\sigma^k)$ and $(\Sigma^+_2,\sigma)$ (resp. $(\Sigma_2,\sigma)$ if $X$ is two-sided subshift) are conjugate. Then result follows by Theorem~\ref{thm:full_shift}.
\end{proof}

The following fact may be known, however we could not provide any reference containing it, hence a proof is included.

\begin{prop}
Every totally transitive synchronizing subshift $(X,\sigma)$ is mixing.
\end{prop}
\begin{proof}
Let $s$ be a synchronizing word.
First note that since periodic points are dense, $(X,\sigma)$ is weakly mixing.
Fix any two open sets $U,V$. There are words $u,v$ such that $C_X[u]\subset U$, $C_X[v]\subset V$. We can also find words $w_1,w_2,w_3,w_4$ such that
$u's=s w_1 u w_2 s\in L(X)$ and $v's=s w_3 v w_4 s\in L(X)$. Denote $k=|v'|$ and let $p_i$, $i=0,...,k$ be words provided by weak mixing such that
$|p_{i+j}|=|p_i|+j$ and $s p_i s \in L(X)$. Then for any $j\in \N$ and $i=0,\ldots,k$ the word of the form
$u' s p_i (v')^j$
is in $L(X)$ and $|u' s p_i (v')^j|=|u'|+|s|+|p_0|+i+jk$ which shows that $N(U,V)$ is co-finite. The proof is finished.
\end{proof}

\begin{thm}\label{MixSynchro}
Let $(X,\sigma)$ be a mixing one-sided or two-sided synchronizing subshift, let $z,z'\in X$ be distinct, and let $\delta>0$. Then there is a dense Mycielski syndetically scrambled set $S$ for $\sigma$ with the following extra properties:

\noindent (i) $\sigma^k(S)\subset S$ for some $k\in \mathbb N$.

\noindent (ii) For any two distinct $x,x'\in S$, there is an infinite set $M\subset \mathbb N$ such that $d(z,\sigma^n(x))<\delta$ and $d(z',\sigma^n(x'))<\delta$ for every $n\in M$.
\end{thm}
\begin{proof}
Let $m\in \mathbb N$ be large so that $d(x,y)<\delta$ whenever $x,y\in X$ are with $x_{[0,m]}=y_{[0,m]}$. Let $u=z_{[0,m]}$ and $u'=z'_{[0,m]}$. Also let $s\in L(X)$ be a synchronizing word. Since $(X,\sigma)$ is mixing, there are $p,p',q,q',r,r'\in L(X)$ such that $|p|=|p'|$, $|q|=|q'|$, $|r|=|r'|$ and $spuqu'rs, sp'u'q'ur's\in L(X)$. Define two words $B,C\in L(X)$ as $B=spuqu'r$ and $C=sp'u'q'ur'$. Then $|B|=|C|$, $B\ne C$, and any finite concatenation of $B$'s and $C$'s belong to $L(X)$.
By the same arguments, for any word $w\in L(X)$ there are words $w',w''$ such that $sw'ww''s\in L(X)$ and both $|sw'ww''|$, $|ww''|$ are multiplies of $|B|$.
Similarly to the proof of Theorem~\ref{SFTperiod}, using the words $B,C$ we first construct a syndetically scrambled set, and then replacing some blocks, we modify it to a $\sigma^k$-invariant dense Mycielski syndetically scrambled set, where $k=|B|=|C|$.

Note that if $T$ is a syndetically scrambled set contained in $\{B,C\}^\mathbb N\cup \{0\}$, then for any two distinct $x,x'\in T$, there is an infinite set $M\subset \mathbb N$ such that either $(\sigma^n(x)_{[0,k)},\sigma^n(x')_{[0,k)})=(B,C)$ for every $n\in M$,
or $(\sigma^n(x)_{[0,k)},\sigma^n(x')_{[0,k)})=(C,B)$ for every $n\in M$. In either case we note by the definition of $B,C$ that there is an infinite set $M'\subset \mathbb N$ such that $(\sigma^n(x)_{[0,m]},\sigma^n(x')_{[0,m]})=(u,u')$ for every $n\in M'$. That is, $d(z,\sigma^n(x))<\delta$ and $d(z',\sigma^n(x'))<\delta$ for every $n\in M'$.
The proof is completed.
\end{proof}

\begin{thm}\label{Synchro-factor}
Let $(X,\sigma)$ be an infinite one-sided or two-sided mixing synchronizing subshift, let $(Y,g)$ be a dynamical system, and let $\pi\colon (X,\sigma)\to (Y,g)$ be a countable-to-one factor map. Then $(Y,g)$ has a dense Mycielski syndetically $\epsilon$-scrambled set which is $g^k$-invariant for some $k\in \mathbb N$.
\end{thm}
\begin{proof}
Note that $X$ is an uncountable perfect set by hypothesis. Since $\pi$ is countable-to-one, we can choose two points $z_1,z_2\in X$ such that $\pi(z_1)\ne \pi(z_2)$. Put $\eps=d(\pi(z_1),\pi(z_2))/3$. Now apply Theorem \ref{MixSynchro} and Proposition \ref{countable-to-one}.
\end{proof}

\subsection{Coded subshifts}

As we mentioned earlier, coded subshifts form the most general class of transitive subshifts considered so far. Unfortunately, the description of coded subshift has a few limitations when compared to synchronized subshifts. The following example highlights the disadvantages. Namely, it may happen that the SFT's approximating a mixing coded subshift $X_W$ may not mixing.

\begin{exmp}\label{ex:coded_mix}
There exists a sequence of words $\mathcal{B}=\set{B_1,B_2,\ldots}$ such that $B_1,\ldots, B_n$ is a circular code for any $n>1$, and if we denote by $X_n$ the coded system defined by these words (which is a transitive SFT), then $\bigcup_{i=1}^\infty X_i$ is dense in $\Sigma_2$ while $\sigma|_{X_n}$ is not mixing for any $n\in \N$.
\end{exmp}
\begin{proof}
Let $u(1),u(2),\ldots$ be a listing of all nonempty words of even length over $\set{0,1}$ so that $|u(n)|\le 2n$ and put $w(n)=1u(n) 10^{4n}$

By the above construction:
\begin{enumerate}[(i)]
\item if we consider any word $w$ which is a concatenation of words $w(1),\ldots, w(n)$ then it cannot contain $0^{4(n+1)}$ as a subword, since each of the words $w(i)$ starts from symbol $1$ and all of them have length bounded from above by $4n+2n+2<4(n+1)$,
\item the position of $0^{4n}$ in $w(n)$ is uniquely determined.
\end{enumerate}
Therefore, if we consider any bi-infinite concatenation of words $w(1),\ldots, w(n)$ then positions of words $w(n)$ in this bi-infinite sequence are uniquely determined by blocks of the form $0^{4n}$. When we recognize positions of words $w(n)$ then on remaining ``free'' positions we can uniquely determine places occupied by words $w(n-1)$ etc.
In other words, the finite sequence $w(1),\ldots, w(n)$ is a circular code for any fixed $n$.

Hence the coded system $X_n$ generated by $w(1),\ldots, w(n)$ is a transitive SFT, and
$$
L(\overline{\bigcup_{n=1}^\infty X_n})=\bigcup_{n=1}^\infty L(X_n)=\set{0,1}^+
$$
thus $X_W=\Sigma_2$, where $W=\set{w(n):n\in \N}$.

But none of the subshifts $X_n$ can be mixing, since if $x\in X_n$ contains $0^{4n}$ as a subword starting on positions say $i<j$ then $j-i=2 m$ for some integer $m$. Simply $0^{4n}$ can appear only in words $w(n)$ (and the position of this block within $w(n)$ is unique) and between them in $x$ we can have only a concatenation of words $w(1),\ldots, w(n)$ which by definition has always even length.
\end{proof}

Observe that if $X,Y$ are finite and transitive subshifts and $X\subset Y$ then $X=Y$.
Thus if $X_W$ is an infinite coded subshift, then there exists an infinite and transitive SFT such that $Y\subset X_W$. This by Corollary~\ref{cor:inf_trans_SFT}, immediately leads to the following.

\begin{cor}
Let $(X,\sigma)$ be an infinite coded subshift. Then for some $k\in \N$ there exists a $\sigma^k$-invariant Mycielski syndetically scrambled set for $(X,\sigma)$.
\end{cor}

In view of Example~\ref{ex:coded_mix} we see that even if $X_W$ is mixing, then the circular codes coming from $W$ can define nonmixing shifts of finite type.
Still, we strongly believe (however do not know how to prove it formally) that there exists an alternative choice of list of words $\mathcal{B}$ such that $X_\mathcal{B}=X_W$ and the assumptions of Proposition~\ref{prop:mixing_coded} are satisfied.

\begin{prop}\label{prop:mixing_coded}
Let $(X,\sigma)$ be an infinite mixing coded subshift and suppose that there are mixing SFT $X_1\subset X_2\subset\cdots \subset X$ such that $X=\overline{\bigcup_{i=1}^\infty} X_i$. Then there is a dense Mycielski syndetically scrambled set for $(X,\sigma)$.
\end{prop}
\begin{proof}
We may assume that $X_n$ is infinite for every $n\in \mathbb N$. Let $x\in X_1$ be a $\sigma$-periodic point. Then $x\in X_n$ for every $n\in \mathbb N$. By Theorem~\ref{SFTperiod}, for each $n\in \mathbb N$, there is a Mycielski syndetically scrambled set $S_n\subset X_n$ for $\sigma$ with $\overline{S_n}=X_n$ and $x\in S_n$. Put $S=\bigcup_{n=1}^\infty S_n$. Then $S$ is a dense Mycielski set in $X$. Moreover $S^2\subset \SProx(\sigma)$ since $x$ is a common point for all $S_n$'s and the syndetically proximal relation is an equivalence relation. Now consider a nonempty open set $U\subset X$. Then there is $n\in \mathbb N$ such that $X_n\cap U\ne \emptyset$ and hence $S_n\cap U\ne \emptyset$ because $\overline{S_n}=X_n$. An infinite mixing subshift cannot have isolated points and therefore $S_n\cap U$ is an uncountable Borel set. Hence there is a Cantor set $C\subset S_n\cap U$. Clearly $C$ is syndetically scrambled for $\sigma$ since $S_n$ is. Also recall that for a subshift, any scrambled set is $1$-scrambled. Thus the assumptions of Proposition~\ref{joinCantScra} are satisfied for the set $S$ with $\eps=1/3$. Hence there is a dense Mycielski syndetically scrambled set $T\subset X$ for $\sigma$ with $T\subset S$.
\end{proof}

\begin{prop}\label{prop:asc_prox}
Let $X$ be an infinite compact metric space, let $(X,f)$ be a dynamical system and assume that there exists a sequence of closed invariant sets
$\set{X_n}_{n=1}^\infty  \subset X$ such that $\overline{\bigcup_{n=1}^\infty X_n}=X$, $X_i\cap X_j\neq \emptyset$ for $i\neq j$ and $f|_{X_n}$ is a transitive proximal system for every $n$.
Then there exists a dense and $f$-invariant Mycielski syndetically scrambled set $S\subset X$ for $f$.
\end{prop}
\begin{proof}
First of all, if each $X_n$ is finite then $X$ is a singleton, thus without loss of generality we may assume that there is a nonempty set $\mathbb{I}\subset \N$
such that $X_n$ is infinite and $X_i\neq X_j$ for every $n,i,j\in \mathbb{I}$. Note that each $X_n$, $n\in \mathbb{I}$ is perfect, therefore $X$ is perfect (and $\overline{\bigcup_{n\in \mathbb{I}}X_n}=X$) or singleton (and $\mathbb{I}=\emptyset$). It is also easily seen that there is a fixed point $p\in X$ such that $p\in \bigcap_{n=1}^\infty X_n$.

Fix any $n\in \mathbb{I}$. Since $f|_{X_n}$ is transitive, it is well known (e.g. see \cite{Akin}) and not hard to verify that $\Rec(f\times f)\cap X_n\times X_n$ is residual in $X_n\times X_n$.
This implies by Kuratowski-Mycielski theorem that there is a Cantor set $C\subset \Tran(f|_{X_n})$ such that $C\times C\subset \Rec(f\times f)$.
Denote $M_n=\bigcup_{i=1}^\infty f^i(C)$. Note that if $x,y\in C$ are distinct then $f^i(x)\neq f^i(y)$ for every $i$, as otherwise $(x,y)\not\in \Rec(f\times f)$. This shows that $M_n$ is a Mycielski set, and also $\overline{M_n}=X_n$ since $M_n\cap \Tran(f|_{X_n})\neq \emptyset$.

Denote $S=\bigcup_{n\in \mathbb{I}}M_n\cup \set{p}$ and observe that it is possible to construct a Cantor subset of $S$ containing $p$. Therefore $S$ is a dense Mycielski set and also $f(S)\subset S$.

Next fix any distinct $x,y\in S$. If $x,y\in M_n\cup \set{p}$ for some $n$ then $(x,y)\notin \Asy(f)$. Simply $\Rec(f\times f)\cap \Asy(f)\subset \Delta$ and if $(u,w)\in \Rec(f\times f)$ then also $(f^i(u),f^j(w))\in \Rec(f\times f)$ for any $i,j\geq 0$.

There is also a second possibility that there are $n\neq m$ such that $x\in M_n$ and $y\in M_m$. Since $X_n\neq X_m$ we may assume that there is an open set $U\subset X_m$ such that $\overline{U}\cap X_n=\emptyset$. Orbit of $y$ is dense in $X_m$ so it must visit $U$ infinite number of times. This proves that $(x,y)\notin \Asy(f)$ also in this case, hence $S$ is a syndetically scrambled set. The proof is completed.
\end{proof}

Fix $P\subset \N$ and let $X_P$ be the subset of $\Sigma_2$ consisting of all sequences $x$ such that if $x_i=x_j=1$ for some $i,j$ then $|i-j|\in P\cup \set{0}$. It is easy to see that $X_P$ is a subshift, i.e. it is closed and $\sigma$-invariant. We will write $\sigma_P$ for $\sigma$ restricted to $X_P$, and call the dynamical system given by $\sigma_P\colon X_P\to X_P$ a \emph{spacing} shift. One-sided spacing shifts are defined similarly. The class of spacing shifts was introduced by Lau and Zame in \cite{LZ}, and for a detailed exposition of their properties we refer to \cite{Spacing}. For the purposes of this article it is enough to recall that (i) $\sigma_P$ is weakly mixing iff $P$ is thick, and (ii) if $\N\setminus P$ is thick, then $\sigma_P$ is proximal.

\begin{cor}
Let $(X_P, \sigma_P)$ be a weakly mixing spacing shift. Then there exists a dense and invariant Mycielski syndetically scrambled set $S\subset X_P$.
\end{cor}
\begin{proof}
Since $P$ is thick, there exists a thick set $Q\subset P$ such that $\N\setminus Q$ is also thick. For $n\in \N$ let $Q(n)=Q\cup ([0,n]\cap P)$.
Obviously $\bigcup_{n=1}^\infty Q(n)=P$ and so the sequence $X_n=\Sigma_{Q(n)}$ satisfies assumptions of Proposition~\ref{prop:asc_prox}, which ends the proof.
\end{proof}

\begin{rem}
It is very easy to see that $X_P$ has a scrambled pair iff it is uncountable. It is also very easy to change a point $x\in X_P$ with infinitely many occurrences of symbol $1$ into a Cantor syndetically scrambled set (simply, it is enough to change appropriate symbols $1$ to $0$). Therefore, $X_P$ is uncountable iff it has a Cantor syndetically scrambled set.
\end{rem}

Notice that full shift (one or two-sided) is the spacing shift defined by $P=\N$. Then we get another proof of Theorem~\ref{thm:full_shift} in the restricted case of $m=2$.

\subsection{Application to topologically Anosov maps}

Recall that a surjective continuous map  $f\colon X\to X$ is \emph{c-expansive (with expansive constant $\beta>0$)} if for any two sequences
$\set{x_i}_{i\in \Z}, \set{y_i}_{i\in \Z}\subset Y$ such that $f(x_i)=x_{i+1}$, $f(y_i)=y_{i+1}$ and $x_0\neq y_0$ there is $n\in \Z$ such that $d(x_n,y_n)>\beta$. A sequence $\set{x_i}_{i=0}^\infty$ is a $\delta$-pseudo-orbit if $d(f(x_i),x_{i+1})<\delta$ for every $i$. We say that $\delta$-pseudo-orbit $\set{x_i}_{i=0}^\infty$ is $\eps$-traced by a point $z$ if $d(f^i(z),x_i)<\eps$ for every $i\in \N_0$. Recall that $(X,f)$ has \emph{shadowing property} if for every $\eps>0$ there is $\delta>0$ such that any $\delta$-pseudo-orbit is $\eps$-traced by some point. We say that continuous surjection of a compact metric space is a \emph{topologically Anosov map (TA-map for short)} if it is c-expansive and has shadowing. If additionally $f$ is invertible then we call it a \emph{TA-homeomorphism}.

The following fact was first proved by Bowen (see \cite[Theorems~4.3.5~and~4.3.6]{AH}).

\begin{thm}\label{Bowen}
Let $g\colon Y\to Y$ be a topologically mixing TA-homeomorphism. Then there exist a two-sided mixing subshift of finite type $(X,\sigma)$, $t\in \mathbb N$, and a factor map $\pi \colon (X,\sigma) \to (Y,g)$ such that $\# \pi^{-1}(y)\le t$ for every $y\in Y$.
\end{thm}

So it is natural to expect that we may be able to transfer syndetically scrambled set to a TA-map via the factor map. Indeed, by the theory we have developed in the previous sections, we can do it.

\begin{cor}\label{cor:TA-map}
Let $g\colon Y \to Y$ be a mixing TA-map. Then there are $\eps>0$, $k\in \mathbb N$, and a dense Mycielski set $T\subset Y$ such that $T$ is a $g^k$-invariant syndetically $\eps$-scrambled set for $g$.
\end{cor}
\begin{proof}
If $g$ is TA-homeomorphism then the result follows by Theorems~\ref{Synchro-factor} and \ref{Bowen}.
If $g$ is only a TA-map, that is $g$ is non-invertible, consider the inverse limit
$$
\mathbb{X}_g=\lim_{\longleftarrow}(X,g)=\set{(x_0,x_1,\ldots)\; : \; g(x_{n+1})=x_n}\subset X^{\N_0}
$$
together with the induced shift map $\sigma_g\colon \mathbb{X}_g \to \mathbb{X}_g$, $\sigma_g(x)_i=g(x_{i})$.

Then it is well known that $\sigma_g$ is a homeomorphism ($\mathbb{X}_g$ is a compact metrizable space when endowed with Tychonoff topology) which has shadowing
if and only if so does $g$. Thus $\sigma_g$ is TA-homeomorphism, in particular has a dense Mycielski $\sigma_g^k$-invariant syndetically scrambled set $D$.
It is not hard to verify that the natural projection $\pi\colon \mathbb{X}_g \ni x \mapsto x_0 \in X$ preserves syndetically scrambled sets, that is $S=\pi(D)$ is also a syndetically scrambled set. By the definition we also have $g^k(S)\subset S$ which ends the proof.
\end{proof}

\begin{rem}
Let $u,w$ be two words of co-prime length forming a circular code. Then coded system $X$ generated by $u,w$ is a mixing SFT, in particular $\sigma|_X$ is a mixing TA-homeomorphism. But then
$$
k=\min\set{|u|,|w|}=\min\set{p\in \N\; : \; f^p(x)=x, x\in X}
$$
is the smallest number that can be provided by Corollary~\ref{cor:TA-map}, since if $f^p(S)\subset S$ and $S\subset X$ is a scrambled set, then $X$ contains also a point of period $p$.
\end{rem}

\section{Substitutions and minimal subshifts}\label{sec:minimal}

Till now we were studying subshifts with highly non-minimal dynamics. Now we are going to consider the case of minimal subshifts.
First note that by Corollary~\ref{cor:syprox} in many standard classes of minimal subshifts (e.g. Toeplitz flows) there is no difference between
proximal and syndetically proximal pairs.

Here we are going to consider subshifts arising from substitutions. While many of them (such as the Toeplitz flows) are almost 1-1 extensions of odometers, as we will see, still there is large class of substitutions which satisfy $\Prox(\sigma|_X)\neq \SProx(\sigma|_X)$.

First, let us briefly recall what subshifts generated by substitutions are.
Let $A$ be a finite alphabet with $|A|\ge 2$.
Recall that a \emph{substitution} is a map ƒ$\tau \colon A \ra A^+$; $\tau$ is a substitution of constant length $p$, $p \geq 2$, if $|\tau(a)| = p$ for any $a\in A$; $\tau$ is primitive if there is
$n\in \N$ such that the symbol $a$ appears in ƒ$\tau^n(b)$ for every $a, b \in A$.
We can easily extend $\tau$ form $A$ to $A^+$ by concatenation, and so we can speak about iterates of $\tau$, i.e. $\tau^n(a)=\tau(\tau^{n-1}(a))$.
The substitution $\tau$ generates a subshift
$X_\tauƒ\subset A^\Z$ that is the smallest subshift of $A^\Z$ admitting all words $\set{ƒ\tau^n(a): n\in \N, a\in A}$. When
$\tau$ is primitive, the subshift $X_\tau$ is minimal. For a more detailed exposition of substitutions the reader is referred to \cite{Q}.

Let $\tau\colon A\to A^p$ be a substitution of constant length $p\ge 2$. We say that $\tau$ has \emph{overall coincidence} if for any two $a,b\in A$, there are $t\in \mathbb N$ and $0\leq i <p^t$ such that $\tau^t(a)_i=\tau^t(b)_i$. Note that if $\tau$ has overall coincidence, then we can find one global $t$ good for any two $a,b\in A$. The following simple lemma shows that the same can be done also for $i$.

\begin{lem}\label{lem:coincidences}
Let $\tau\colon A\to A^p$ be a substitution of constant length $p\ge 2$, and suppose that $\tau$ has overall coincidence.
There are $t\in \N$, $0\leq i < p^t$ and $e\in A$ such that $\tau^t(a)_i = e$ for every $a\in A$.
\end{lem}
\begin{proof}
The proof is an easy induction on the number of letters in $A$ we are comparing at the same time.
First, by the definition of coincidence, if we fix any $a_1,a_2\in A$ then there are $t\in \N$, $0\leq i < p^t$ and $e_2\in A$ such that $\tau^t(a_1)_i=\tau^t(a_2)_i = e_2$.
Now, fix any $n<\# A$, any $n$ distinct letters $a_1,\dots, a_n \in A$ and assume that there are $t\in \N$, $0\leq i < p^t$ such that $\tau^t(a_j)_i=e_n$ for some $e_n\in A$ and all $1\leq j \leq n$. In other words $\tau^t(a_j)=x_j e_n y_j$ where $|x_j|=i$, $|y_j|=p^t-i-1$.

Now, fix any $a_{n+1}\in A$ and denote $e_n'=\tau^t(a_{n+1})_i$, say $\tau^t(a_{n+1})=x_{n+1}e_n' y_{n+1}$ where by the definition $|x_{n+1}|=i$. There are $t'\in \N$, $0\leq i' < p^{t'}$ and $e\in A$ such that $\tau^{t'}(e_{n})_{i'}=\tau^{t'}(e_{n}')_{i'}=e$

Now if we denote $m=ip^{t'}+i'$ then ($j\leq n$):
\begin{eqnarray*}
\tau^{t'+t}(a_j)_m&=& (\tau^{t'}(x_j)\tau^{t'}(e_n)\tau^{t'}(y_j))_m=\tau^{t'}(e_n)_{i'}=e\\
\tau^{t'+t}(a_{n+1})_m&=& (\tau^{t'}(x_{n+1})\tau^{t'}(e_n')\tau^{t'}(y_{n+1}))_m=\tau^{t'}(e_{n}')_{i'}=e.
\end{eqnarray*}
and so the result follows.
\end{proof}

This shows that the definition of coincidence from \cite{Q} and the definition of overall coincidence given above (which comes from \cite{BDM}) coincide.
Given a primitive substitution of constant length $\tau\colon A \ra A^+$ we define (after \cite{Kamae}) the \emph{column number} of $\tau$ by:
$$
C(\tau)=\min_{k\in \N}\min_{0\leq j < p^k}\# \set{\tau^k(a)_j : a\in A}.
$$

We say that a finite sequence $X_0,\ldots, X_n$ of \emph{regularly closed sets} (i.e. $\overline{\Int X_j}=X_j$) is a regular periodic decomposition of $X$
if $\Int X_j \cap X_i=\emptyset$ for $i\neq j$, $X=\bigcup_j X_j$ and $f(X_j)=X_{j+1}$ for $0\leq j <n$ and $f(X_n)=X_0$. By a \emph{periodic $f^m$ minimal partition} of
$X$ we mean a regular periodic decomposition $X_0,\ldots, X_{m-1}$ of $X$ whose elements are minimal for $f^m$. It is not hard to see that indeed it is a partition, that is the condition $X_j\cap X_i= \emptyset$ must be satisfied for any $i\neq j$.
The equivalence relation whose classes are the members of periodic $f^n$ minimal partition is denoted by $\Lambda_n$ and \emph{the trace relation} is defined by $\Lambda=\bigcap_{n\in \N}\Lambda_n$. For any $x\in X_\tau$ denote by $[x]_\Lambda$ the equivalence class of $x$ in $\Lambda$.

Combining the main result of \cite{Kamae} (that is \cite[Theorem~5]{Kamae}) with Lemma~13 and Theorem~2 from the same paper we get the following:
\begin{thm}\label{thm:sinle_fib_subst}
Let $\tau\colon A \ra A^p$ be a primitive substitution a constant length $p\geq 2$.
Then $\min_{x\in X_\tau} \# [x]_\Lambda \leq C(\tau)$.
\end{thm}

\begin{thm}
Let $\tau\colon A\to A^p$ be a substitution of constant length $p\ge 2$, and suppose that $\tau$ has overall coincidence.
Then $\Prox(\sigma|_{X_\tau})=\SProx(\sigma|_{X_\tau})$.
\end{thm}
\begin{proof}
If $X_\tau$ is finite then it is periodic orbit and so there is nothing to prove. For the case of $X_\tau$ infinite,
by Lemma~\ref{lem:coincidences} and Theorem~\ref{thm:sinle_fib_subst} we see that at least one class of $\Lambda$ is a singleton.
It is not hard to verify that if $[x]_\Lambda$ is a singleton, then $x$ is \emph{regularly recurrent}, that is for any open set $U\ni x$ there is $k$
such that $f^{kj}(x)\in U$ for all $j\geq 0$. But then we can use \cite[Theorem~5.1]{Downar} which says that any system defined by the orbit closure of
a regularly recurrent point is an almost 1-1 extension of an odometer. It remains to apply Corollary~\ref{cor:syprox} and the proof is completed.
\end{proof}

Using \cite[Proposition~3.5]{BDM} we see that there are numerous substitutions with syndetically scrambled pairs.
\begin{exmp}
Consider the substitution on $A=\set{0,1}$ given by $\tau(0)=001$, $\tau(1)=100$. If we denote $f=\sigma|_{X_\tau}$ then $f$ has scrambled pairs
and $\Prox(f)=\SProx(f)$ (so every scrambled pair is syndetically scrambled).
\end{exmp}

Let $f\colon X \ra X$ be a continuous map on a compact metric space. Recall that a pair $(x,y)$ is \emph{DC1} if:
$\Phi^*_{xy}(t)=1$ for every $t>0$ and $\Phi_{xy}(s)=0$ for some $s>0$ where
\begin{eqnarray*}
\Phi^*_{xy}(t)&=&\limsup_{n\ra\infty}\frac{1}{n}\# \set{i\leq n : d(f^i(x),f^i(y))<t},\\
\Phi_{xy}(t)&=&\liminf_{n\ra\infty}\frac{1}{n}\# \set{i\leq n : d(f^i(x),f^i(y))<t}.\\
\end{eqnarray*}

The following observation was first used in \cite{DCRev} to show that proximal system can never have DC1 pairs.

\begin{rem}
If $(x,y)\in \SProx(f)$ then it is not DC1, since for every $s>0$ there is $M\in \N$ such that for every $n\geq 0$ and $0\leq r<M$
$$
\frac{1}{nM+r}\# \set{i\leq Mn+r : d(f^i(x),f^i(y))<s}\geq \frac{n}{Mn+r}\geq \frac{1}{2M}
$$
and as a consequence $\Phi_{xy}(s)>0$ for every $s>0$.
\end{rem}

It was proved in \cite{hwfl} that there are substitutions with DC1 pairs, thus in all such cases $\Prox(\sigma|_{X_\tau})\neq \SProx(\sigma|_{X_\tau})$.
The natural question is whether it is the full classification of substitutions with proximal but not syndetically proximal pairs.
The following example shows that the answer to this question must be in the negative.

Before we proceed let us recall two definitions from \cite{hwfl}.
A pair of words $u,v\in A^+$ is \emph{mutually exclusive} (denoted by $u \vee v$) if $\tau^k(u)_i\neq \tau^k(v)_i$ for each $k,i$. If for any choice of $k,i$
letters $\tau^k(u)_i,\tau^k(v)_i$ are not mutually exclusive, we say that $u,v$ are \emph{mutually attractive} (denoted by $u \wedge v$).

\begin{exmp}
Let $A=\{a,b,c,d\}$.
Consider the substitution of constant length $\tau\colon A \to A^3$ given by:
$$
\tau(a)=aab, \quad  \tau(b)=bad,\quad \tau(c)=ccd,\quad \tau(d)=dcb.
$$
Then $\tau$ is primitive, $\Prox(\sigma|_{X_\tau})\neq \SProx(\sigma|_{X_\tau})$ but $\tau$ does not have DC1 pairs.
\end{exmp}
\begin{proof}
First observe that $\tau$ is primitive.
Next, note that if we fix any $n\in \N$ and $0\leq i < 3^n$ then $\set{\tau^n(a)[i],\tau^n(c)[i]}$ is either $\set{a,c}$ or $\set{b,d}$ and the same holds for
$\set{\tau^n(b)[i],\tau^n(d)[i]}$. In particular $\tau^n(b)[i]\neq \tau^n(d)[i]$ for any $n\in \N$ and $0\leq i < 3^n$.
Now if we consider points
$$
x^a=\lim_{n\ra\infty}\tau^n(a),\quad x^b=\lim_{n\ra\infty}\tau^n(b)
$$
then $(x^a,x^b)$ is obviously a proximal pair since $\tau(a)[1]=\tau(b)[1]=a$ and thus $x^a_{[3^n,3^n 2)}=x^b_{[3^n,3^n 2)}=\tau^n(a)$.
But it cannot be syndetically proximal, since $\tau(a)[2]=b$ while $\tau(b)[2]=d$ and so $x^a_{[3^n 2,3^{n+1})}=\tau^n(b)$, $x^a_{[3^n 2,3^{n+1})}=\tau^n(d)$
which implies
$$
\set{i: d(f^i(x^a),f^i(x^b))>1/2}\supset \N \cap \bigcup_{i}[3^i 2,3^{i+1}).
$$
It remains to prove that there is no DC1 pair in $X_\tau$.
By \cite[Theorem~3.3]{hwfl} we see that if there are DC1 pairs then there must exist letters $\gamma_1,\gamma_2,\beta_1, \beta_2$
such that for some $n\in \N$ we have
$$
\tau^n(\beta_i)=C_i\gamma_i\eta_i D_i \beta_i E_i,\quad \tau^n(\gamma_i)=F_i\beta_i\delta_i G_i \gamma_i H_i
$$
where $|C_1|=|C_2|$, $|D_1|=|D_2|$, $|F_1|=|F_2|$, $|G_1|=|G_2|$, $E_1 \wedge E_2$ and $H_1\vee H_2$, $\eta_1\vee \eta_2$, $\delta_1 \wedge \delta_2$.
(here $\delta_i, \eta_i\in A$ and words $E_i,H_i\in A^+$, while others can be empty words). It is obvious from the definition that $\beta_1\neq \beta_2$ and $\gamma_1\neq \gamma_2$.

First note that for any letter $p\in A$
the word $\tau^n(p)$ ends with $b$ or $d$. Another immediate observation is that the last letters of $\tau^n(\gamma_i)$ cannot be the same. Since $\tau^n(a)[i]\neq \tau^n(c)[i]$ and $\tau^n(b)[i]\neq \tau^n(d)[i]$ for any choice of $n,i$ the only possibility is that
$$
(\gamma_1,\gamma_2)\in \set{(c,d), (a,b)}.
$$

First suppose $(\gamma_1,\gamma_2)=(c,d)$. Observe that if $(s,t)$ is a pair in $A$ such that $(\tau(s)[i],\tau(t)[i])=(c,d)$, then $(s,t)=(c,d)$ and $i=1$, that is, $(c,d)$ must appear as the prefix of $(\tau(s),\tau(t))$. By induction, we get that if $(s,t)$ is a pair in $A$ with $(\tau^n(s)[i],\tau^n(t)[i])=(c,d)$, then $(s,t)=(c,d)$ and $i=1$. But $\tau^n(\gamma_i)=F_i\beta_i\delta_i G_i \gamma_i H_i$ which forces the entire block $F_i\beta_i\delta_i G_i$ to become the empty word, a contradiction since $\beta_i$ is a letter.

The case $(\gamma_1,\gamma_2)=(a,b)$ is also ruled out by a similar argument. So we see that equivalent conditions form \cite[Theorem~3.3]{hwfl} cannot be satisfied. This proves that there is no DC1 pair in $X_\tau$ and so the proof is completed
\end{proof}

\section{A map without nontrivial syndetically proximal pairs}\label{CounterEg}

An immediate question, when looking at the definitions is whether the existence of (non-diagonal) proximal pairs forces syndetically proximal pairs to exist.
This question was explicitly stated in \cite[Remark on p.4]{tksm}. The following example shows that the answer must be in the negative.

\begin{exmp}\label{ex:pros_synd}
There is a compact metric space $X$ and a homeomorphism $f\colon X\to X$ such that $\SProx(f)=\Delta_X$, but $\Prox(f)\ne \Delta_X$.
\end{exmp}
\begin{proof}
Use the notation $e(\theta)=e^{2\pi i\theta}$ for convenience. Let $\alpha\in (0,1)$ be an irrational number and let $g\colon \Si\to \Si$ be the corresponding irrational rotation of the unit circle $\Si\subset \C$ given by $g(x)=e(\alpha)x$. Let $m(0)=0$ and $m(n)=1+2+\cdots+n$ for $n\in \mathbb N$. Define a sequence of points $z_j=(x_j,y_j)\in \Si\times \mathbb R$ inductively as follows.

\begin{eqnarray*}
(x_0,y_0)&=&(1,1)\in \Si\times \R;\\
(x_{j+1},y_{j+1})&=&(e(\alpha/n)g(x_j),1/n), \textrm{ if } m({n-1})\le j<m(n) \textrm{ and } n  \textrm{ is odd};\\
(x_{j+1},y_{j+1})&=&(e(-\alpha/n)g(x_j),1/n),\textrm{ if } m({n-1})\le j<m(n) \textrm{ and }n \textrm{ is even.}
\end{eqnarray*}

We list a few first terms of the sequence to give a quick idea of what is going on:
\begin{eqnarray*}
&&(1,1), (e(2\alpha),1), (e(2\tfrac{1}{2}\alpha),\tfrac{1}{2}), (e(3\alpha),\tfrac{1}{2}),\\
&& (e(4\tfrac{1}{3}\alpha),\tfrac{1}{3}), (e(5\tfrac{2}{3}\alpha),\tfrac{1}{3}), (e(7\alpha),\tfrac{1}{3}),\\
&& (e(7\tfrac{3}{4}\alpha),\tfrac{1}{4}), (e(8\tfrac{2}{4}\alpha),\tfrac{1}{4}), (e(9\tfrac{1}{4}\alpha),\tfrac{1}{4}), (e(10\alpha),\tfrac{1}{4}),\\
&& (e(11\tfrac{1}{5}\alpha),\tfrac{1}{5}),\ldots
\end{eqnarray*}

Using mathematical induction, let us observe the following:

\begin{enumerate}[(a)]
\item $x_{m(n)}=e((m(n)+1)\alpha)=g^{m(n)}(g(1))$, if $n$ is odd.
\item $x_{m(n)}=e(m(n)\alpha)=g^{m(n)}(1)$, if $n$ is even.
\end{enumerate}

For $j<0$ define $z_{j}=(e(j\alpha),2-\frac{1}{1-j})$

Let $X=\{z_j=(x_j,y_j):j\in \Z\}\cup (\Si\times \{0,2\})\subset \mathbb C\times \mathbb R$, and let $d$ be the metric on $X$ induced by the Euclidean metric from $\R^3$. Clearly $X$ is compact since $\lim_{j\ra \infty}y_j=0$ and $\lim_{j\ra -\infty}y_j=2$. Define $F\colon X\ra X$ putting $F(x,0)=(g(x),0)$, $F(x,2)=(g(x),2)$ for $x\in \Si$, and $F(z_j)=z_{j+1}$ for $j\in \Z$. Then we have that:

\begin{enumerate}[(i)]
\item $F$ and its inverse are continuous, since $\lim_{n\ra \infty}\alpha/n= 0$.

\item The element $z_0$ is proximal to both $(1,0)$ and $F(1,0)$ under the action of $F$ (by the definition of $F$ and by previous observations (a) and (b)).

\item $z_j$ is not syndetically proximal to any point of $\Si \times \{0,2\}$ (since any two distinct points of $\Si\times \{0\}$ are distal, and since $z_j$ is proximal to two distinct points of $\Si \times \{0\}$, namely $F^j(1,0)$ and $F^{j+1}(1,0)$.)

\item If $j<k$, then $z_j$ is not syndetically proximal to $z_k$ (since for any $\epsilon>0$, the set $\{n\in \mathbb N:d(F^n(z_j),F^n(F^j(1,0)))<\epsilon \text{ and }d(F^n(z_k),F^n(F^k(1,0)))<\epsilon\}$ is thick by (ii) and the points $F^j(1,0)$, $F^k(1,0)$ form a distal pair).
\end{enumerate}
\end{proof}

\begin{rem}
While it would be nice to have transitivity (or better minimality) in Example~\ref{ex:pros_synd}, it seems that construction of such system can be tough.
At least it must have zero topological entropy, cannot be almost 1-1 extension of distal system, neither can be expansive. Simply, systems with positive topological entropy \cite{Blanchard} as well as well as expansive systems \cite{King} always have non-diagonal asymptotic pairs; distality is excluded by Corollary~\ref{cor:syprox}. This disables possibility of application of most common techniques from topological dynamics (in particular such a system cannot be a subshift).
\end{rem}

\section{Interval maps}\label{sec:interval}

By an interval map we mean a continuous map $f\colon [0,1]\to [0,1]$. Let us divide the class of dynamical systems $([0,1],f)$
into four disjoint subclasses, depending on the dynamical properties of $f$:
\begin{enumerate}[(i)]
\item Mixing interval maps,

\item Transitive but not mixing interval maps,

\item Non-transitive interval maps with $\htop(f)>0$,

\item Interval maps with $\htop(f)=0$.
\end{enumerate}

In this section and the next section, we provide information about the possible existence and the nature of syndetically scrambled sets for maps in these four classes. It is already known that interval maps $f$ with $\htop(f)>0$ have uncountable syndetically scrambled sets invariant under some power of $f$, and also have Cantor syndetically scrambled sets \cite[Theorem~10 and Theorem~11]{tksm}.

\subsection{Transitive interval maps}

In the rest of this section, we consider transitive interval maps (both mixing and non-mixing).

\begin{thm}\label{cdinvss}
Let $([0,1],f)$ be a mixing dynamical system. Then there are $n\in \mathbb N$, $\eps>0$, and a dense Mycielski set $T\subset [0,1]$ such that $T$ is an $f^{2^n}$-invariant syndetically $\eps$-scrambled set for $f$.
\end{thm}

\begin{proof}
It is easy to verify that any mixing interval map has positive topological entropy. Then using
\cite[Theorem~9]{tksm} we can find $n\in \mathbb N$, an $f^{2^n}$-invariant closed set $X \subset [0,1]$ and a conjugacy $\phi\colon (X,f^{2^n}) \ra (\Sigma_4,\sigma)$. It is an immediate consequence of the construction of $X$ in the proof of \cite[Theorem~9]{tksm} that there are pairwise disjoint closed intervals $J_0,J_1,J_2,J_3\subset [0,1]$ such that $\phi(X\cap J_i)=C[i]$ where $C[i]$ denotes the cylinder set $C[i]=\set{y\in \Sigma_4\; : \; y_0=i}$. Without loss of generality we may assume that the four intervals $J_0,J_1,J_2,J_3$ appear in that order from left to right inside of the interval $[0,1]$ (i.e $x<y$ for all $x\in J_i$, $y\in J_{i+1}$). Then there is $\eps>0$ such that $J_1\cup J_2\subset (\epsilon,1-\epsilon)$ and $\dist(J_1,J_2)>\eps$. Let $\Lambda\subset \{1,2\}^{\N}$ be a Mycielski $\sigma$-invariant syndetically scrambled set for $\sigma$ given by Theorem~\ref{thm:full_shift}, and let $S=\phi^{-1}(\Lambda)\subset J_1\cup J_2\subset [0,1]$. Note that, since $\dist(J_1,J_2)>\eps$, $S$ is a Mycielski $f^{2^n}$-invariant syndetically $\eps$-scrambled set for $f^{2^n}$, with the additional property $S\subset (\eps,1-\eps)$. Write $g=f^{2^n}$. If $U\subset [0,1]$ is any nonempty open set, then $f^m(U)\supset (\epsilon, 1-\eps)\supset S\supset g^m(S)$ for all $m$ large enough. Now it is enough to apply Lemma~\ref{InvMyc} to the map $g$ and the proof is finished.
\end{proof}

We see that the assumption of mixing plays an important role in the above proof. If $f$ is transitive but not mixing, then it is well known that $[0,1]$ can be decomposed into two intervals $J_0=[0,x]$, $J_1=[x,1]$ with common endpoint $x\in (0,1)$ such that $f(J_0)=J_1$, $f(J_1)=J_0$ and $f|_{J_i}$ is mixing for $i=0,1$ (e.g. see Theorems~7.1 and 7.2 in \cite{Banks}). Let us start with two simple observations which will explain why we need more sophisticated techniques to deal with this case.

\begin{prop}\label{prop:1}
Let $([0,1],f)$ be transitive but not mixing and let $x\in (0,1)$ be the unique fixed point of $f$.
Let $y\in [0,x)$ and $z\in (x,1]$. Then $(y,z)\in \SProx(f)$ if and only if  $(y,x), (x,z)\in \SProx(f)$.
\end{prop}
\begin{proof}
Note that $f^n(y)$ and $f^n(z)$ should be on different sides of $x$ for all $n\in \N$. Therefore, for any $\epsilon>0$, we have $\{n\in \N:|f^n(y)-f^n(z)|<\epsilon\}\subset \{n\in \N:f^n(y),f^n(z)\in (x-\epsilon,x+\epsilon)\}$. This gives $(y,z)\in \SProx(f) \Longrightarrow(y,x), (x,z)\in \SProx(f)$. The other implication is clear since the syndetically proximal relation is an equivalence relation \cite{Clay}.
\end{proof}

\begin{prop}\label{uniMin}
Let  $(X,f)$ be a dynamical system and let $x\in \M(f)$.
Then $\overline{\orbp(x,f)}$ is the unique minimal set in the subsystem $(\overline{\orbp(y,f)}, f)$ for any $y\in \SProx(f)(x)$. In particular, if $f(x)=x$, then:
\begin{enumerate}[(i)]
\item $x$ is the only minimal point in $\overline{\orbp(y,f)}$ for any $y\in \SProx(f)(x)$,
\item $\SProx(f)(x)=\bigcup_{y\in \SProx(f)(x)} \overline{\orbp(y,f)}$.
\end{enumerate}
\end{prop}
\begin{proof}
Let $K=\overline{\orbp(x,f)}$ and $y\in \SProx(f)(x)$. Since $\{n\in \N : d(x,f^n(x))<\epsilon\}$ is syndetic and $\{n\in \N : d(f^n(x),f^n(y))<\eps\}$ is thick, we have $x\in \overline{\orbp(y,f)}$, and therefore $K\subset \overline{\orbp(y,f)}$. Now let $z\in \overline{\orbp(y,f)}\cap \M(f)$. Observe that
$\{n\in \N : d(z,f^n(y))<\eps\}$ is piecewise syndetic and $\set{n\in \N:d(f^n(y),K)<\eps}$ is thickly syndetic. This implies $z\in K$, since $\epsilon>0$ is arbitrary.
\end{proof}

From the two Propositions above, we conclude the following. If $f\colon [0,1]\to [0,1]$ is a transitive, non-mixing map with unique fixed point $x\in [0,1]$, then for a dense subset $S\subset [0,1]$ to be a syndetically scrambled set for $f$, it is necessary that $\overline{\orbp(y,f)}\cap (\M(f)\setminus \{x\})=\emptyset$ for every $y\in S$. This puts some restrictions on $S$. However, this restriction is not strong enough to prevent the existence of such a dense syndetically scrambled set $S$ for $f$. Below we describe the construction of $S$ involving a few preliminary steps. Roughly speaking, the basic idea is to first choose an infinite sequence of well placed subintervals with suitable covering relations among them, and then to use these relations to transfer a scrambled set of the shift map into $[0,1]$.

The following fact is well known and easy to prove (e.g. see \cite[Lemma~A.2]{Buzzi}).

\begin{lem}\label{lem:fixpoints}
Let $([0,1],f)$ be mixing and let $f^{-1}(0)=\set{0}$. Then there is a sequence of fixed points of $f$ in $(0,1)$ converging to 0.
\end{lem}

For a continuous map $f\colon [0,1]\to [0,1]$, we introduce some notations to express the covering relations of subintervals. For $K,L\subset [0,1]$, write $K \Longrightarrow L$ if $L\subset f(K)$, and $K \stackrel{f^n} \Longrightarrow L$ if $L\subset f^n(K)$. Obviously, if $J_0 \Longrightarrow J_1 \Longrightarrow \ldots \Longrightarrow J_n$ then there is
$x\in J_0$ such that $f^i(x)\in J_i$ for $i=0,1, \ldots,n$. Also, let $K \stackrel{s,k}{\dashrightarrow} L$ denote the fact that there is a sequence $J_0, J_1,\ldots ,J_{k}$ of subintervals of $[0,1/s]$ such that $J_0=K$, $J_k=L$, and $
K \Longrightarrow J_1 \Longrightarrow \ldots \Longrightarrow J_{k-1} \Longrightarrow L$.

\begin{lem}\label{lem:chains}
Let $([0,1],f)$ be mixing and let $f^{-1}(0)=\set{0}$. Then there exists a sequence of non-degenerate closed intervals $(L_n)_{n=2}^\infty$ such that:

\begin{enumerate}[(i)]
\item\label{L75:i} $L_n\subset (0,1/n]$ and $L_n \Longrightarrow L_n$.

\item\label{L75:ii} There is $k_n\in \mathbb N$ such that $L_n \stackrel{n,k_n}{\dashrightarrow} L_{n+1}$ and $L_{n+1} \stackrel{n,k_n}{\dashrightarrow} L_n$.

\item\label{L75:iii} There are nondegenerate closed and disjoint intervals $H_0,H_1\subset (0,1)$ and $k_1\in \N$ such that $L_2 \stackrel{f^{k_1}}{\Longrightarrow} H_i$ and $H_i \stackrel{f^{k_1}}{\Longrightarrow} L_2$ for $i=0,1$.
\end{enumerate}
\end{lem}
\begin{proof}
Let $(p_n)_{n=2}^\infty$ be a decreasing sequence of fixed points of $f$ in $(0,1)$ with $\lim_{n\ra\infty}p_n=0$ provided by Lemma~\ref{lem:fixpoints}.
Without loss of generality (removing some points from this sequence if necessary) we may assume that $p_n<1/n$ and $f([p_n,1])\cap [0,p_{n+1}]=\emptyset$ for every $n\ge 2$.

Fix any $n\geq 2$. Since $f$ is transitive, there is $x\in (p_n,1]$ and $k\in \N$ such that $f^k(x)\in (0,p_{n+2})$. We may also assume that $f^i(x)\notin (p_n,1]$ for $0<i<k$. Obviously $f^i(x)\neq p_n$ for every $i\leq k$.
Note that $f(x)\in (p_{n+1},p_n)$ and thus $[f(x),p_n]\subset [0,1/n]$ is a non-degenerate interval.
Additionally each $[f^i(x),p_n]$ is non-degenerate interval and thus we have a natural sequence of coverings
$$
[f(x),p_n] \Longrightarrow [f^2(x),p_n] \Longrightarrow \ldots \Longrightarrow [f^k(x),p_n]\supset [p_{n+2},p_{n+1}].
$$
Removing a few first intervals from the above sequence of coverings if necessary, without loss of generality we may assume that $[f(x),p_n]$ is the only interval in this sequence
fully contained in $[p_{n+1},p_n]$. After this modification the sequence constitutes of at least two intervals (so at least one covering relation is still there). Then it is enough to put $L_n=[f(x),p_n]$ since our modification ensures that $L_n \Longrightarrow L_n$.
We can perform the above construction for any $n$, thus each $L_n\subset [p_{n+1},p_n]$ is well defined. But by the method of construction, we get that $L_n \stackrel{n,k_n}{\dashrightarrow} L_{n+1}$.

Fix $n\ge 2$, and let $m\in \N$ be the smallest integer such that $f^m(L_{n+1})\cap (p_n,1]\ne \emptyset$, then $f^j(L_{n+1})\subset [0,p_n]\subset [0,1/n]$ for $0\le j<m$. Since $p_{n+1}\in L_{n+1}$ is a fixed point, $L_n\subset [p_{n+1},p_n]\subset f^m(L_{n+1})$. Thus for each $n\ge 2$, there is $m_n\in \mathbb N$ such that $L_{n+1} \stackrel{n,m_n}{\dashrightarrow} L_n$.

Since $L_n \Longrightarrow L_n$ for every $n\ge 2$, we may assume $k_n=m_n$ after inserting either a certain number of copies of $L_n$ into the beginning of the covering relation $L_n \stackrel{n,k_n}{\dashrightarrow} L_{n+1}$, or a certain number of copies of $L_{n+1}$ into the beginning of the covering relation $L_{n+1} \stackrel{n,m_n}{\dashrightarrow} L_n$. We have just proved \eqref{L75:i} and \eqref{L75:ii}.

Let $H_0,H_1\subset (0,1)$ be any two non-degenerate, disjoint closed intervals. By mixing and the fact that $L_2\subset (0,1)$, there is $k_1\in \mathbb N$ such that $L_2 \stackrel{f^{k_1}}{\Longrightarrow} H_i$ and $H_i \stackrel{f^{k_1}}{\Longrightarrow} L_2$ for $i=0,1$ which completes the proof.
\end{proof}

\begin{lem}\label{lem:chains2ex}
Let $([0,1],f)$ be mixing and let $f^{-1}(\set{0})\supset \set{0,x}$ for some $x\in (0,1]$. Then there exists a nested sequence of non-degenerate closed intervals $L_2 \supset L_3 \supset \ldots$ such that:
\begin{enumerate}[(i)]
\item\label{L76:i} $L_n\subset [0,1/n]$ and $L_n \Longrightarrow L_n$,

\item\label{L76:ii} there is $k_n\in \mathbb N$ such that $L_n \stackrel{n,k_n}{\dashrightarrow} L_{n+1}$ and $L_{n+1} \stackrel{n,k_n}{\dashrightarrow} L_n$.

\item\label{L76:iii} There are nondegenerate closed and disjoint intervals $H_0,H_1\subset (0,1)$ and $k_1\in \N$ such that $L_2 \stackrel{f^{k_1}}{\Longrightarrow} H_i$ and $H_i \stackrel{f^{k_1}}{\Longrightarrow} L_2$ for $i=0,1$.
\end{enumerate}
\end{lem}
\begin{proof}
For $n\ge 2$, we are going to define certain intervals $L_n=[0,b_n]\subset [0,1/n]$ inductively. By transitivity, there is $b_2\in (0,1/2)$ with $f(b_2)>b_2$. Let $L_2=[0,b_2]$ and note that $L_2\subset f(L_2)$. Next, assume that we have defined $L_j=[0,b_j]$ for $j\le n$. By transitivity, we can find $b_{n+1}\in (0, \min \set{b_n,\,1/(n+1)})$ such that $f(b_{n+1})>b_{n+1}$. Let $L_{n+1}=[0,b_{n+1}]$. Observe that $L_{n+1}\subset L_n\subset f(L_n)$ and that $L_{n+1}\subset f(L_{n+1})$. In particular, $L_n \stackrel{n,k}{\dashrightarrow} L_{n+1}$ for every $k\in \mathbb N$.

Fix $n\ge 2$, and let $k\in \N$ be the smallest integer such that $L_n\subset f^{k}(L_{n+1})$. Note that such a number $k$ must exist by transitivity of $f$ and the connectedness of $L_n$, since $f(0)=0$ and $f^s(L_{n+1})\cap (1/2,1)\neq \emptyset$ for some $s\in \N$. Additionally, for any $i$ there is $c$ such that $f^i(L_{n+1})=[0,c]$ and hence $f^i(L_{n+1})\subset [0,1/n]$ for $0\le i<k$ by the definition of $k$. Thus we have $L_{n+1}\Longrightarrow f(L_{n+1})\Longrightarrow \ldots \Longrightarrow f^{k-1}(L_{n+1})\Longrightarrow L_n$ where by the definition, all these intervals are contained in $[0,1/n]$. Thus $L_{n+1} \stackrel{n,k}{\dashrightarrow} L_n$ and so conditions \eqref{L76:i} and \eqref{L76:ii} are satisfied.

Now, we will prove \eqref{L76:iii}. Since $f^{-1}(\set{0})\cap (0,1]\neq \emptyset$ and $f$ is surjective, there are points $y_1,y_2,y_3\in (0,1]$ such that
$y_3\mapsto y_2\mapsto y_1\mapsto 0$ under action of $f$. Since at most one $y_i$ can be equal to $1$, there are $i\neq j$ such that $y_i,y_j\in (0,1)$.
Then we can easily find two nondegenerate closed and disjoint intervals $H_0,H_1\subset (0,1)$ such that $y_i\in H_0$ and $y_j\in H_1$.
Note that $0=f^k(y_i)=f^k(y_j)\in f^k(H_0)\cap f^k(H_1)$ for every integer $k>3$. Therefore, by mixing and connectedness, we can find an integer $k_1\in \N$ such that $[0,1/2]\subset f^k(H_0)\cap f^k(H_1)$ for every $k\geq k_1$. We may also assume that $f^{k_1}(L_2)$ intersects all three connected components of the set $[0,1]\setminus (H_0\cup H_1)$.
Indeed, $L_2 \stackrel{f^{k_1}}{\Longrightarrow} H_0 \stackrel{f^{k_1}}{\Longrightarrow} L_2$ and $L_2 \stackrel{f^{k_1}}{\Longrightarrow} H_1 \stackrel{f^{k_1}}{\Longrightarrow} L_2$ and so the proof is completed.
\end{proof}

\begin{thm}\label{thm:mixna}
Let $([0,1],f)$ be a mixing dynamical system such that $f(0)=0$.
Then there exist $\eps>0$ and a dense Mycielski set $S\subset [0,1]$ such that $S\cup \{0\}$ is a syndetically $\eps$-scrambled set for $f$.
\end{thm}
\begin{proof}
Let intervals $L_n$,$H_0,H_1$ and integers $k_n$ be provided by Lemma~\ref{lem:chains} or Lemma~\ref{lem:chains2ex}, when $f^{-1}(0)=\set{0}$
or $f^{-1}(0)\supsetneq \set{0}$, respectively.

Now fix any $\alpha\in \Sigma_2$. For this $\alpha$ we can build an infinite sequence of covering relations as follows
\begin{eqnarray*}
&&H_{\alpha_0} \stackrel{f^{k_1}}{\Longrightarrow} L_2 \stackrel{2,k_2}{\dashrightarrow} L_{3} \stackrel{2,k_2}{\dashrightarrow} L_2 \stackrel{f^{k_1}}{\Longrightarrow} \\
&&H_{\alpha_1} \stackrel{f^{k_1}}{\Longrightarrow} L_2 \stackrel{2,k_2}{\dashrightarrow} L_{3} \stackrel{3,k_3}{\dashrightarrow} L_4 \stackrel{3,k_3}{\dashrightarrow}  L_3 \stackrel{2,k_2}{\dashrightarrow} L_2 \stackrel{f^{k_1}}{\Longrightarrow} \\
&&\ldots\\
&&H_{\alpha_{n}} \stackrel{f^{k_1}}{\Longrightarrow} L_2 \stackrel{2,k_2}{\dashrightarrow} L_{3} \stackrel{3,k_3}{\dashrightarrow} \ldots \stackrel{n,k_n}{\dashrightarrow} L_{n+1} \stackrel{n,k_n}{\dashrightarrow}  \ldots \stackrel{3,k_3}{\dashrightarrow}  L_3 \stackrel{2,k_2}{\dashrightarrow} L_2 \stackrel{f^{k_1}}{\Longrightarrow} \\
&&\ldots
\end{eqnarray*}
There is at least one point $b\in [0,1]$ that visits all the intervals building the above sequence under iteration as per the covering relation. If $b$ is one such point, let us say for the reminder of this proof that \emph{$b$ has symbolic coding $\alpha$.} Let $\Gamma\subset \Sigma_2$ be a Cantor scrambled set for $\sigma$, and let $D=\{b\in [0,1]:\text{$b$ has symbolic coding $\alpha$ for some }\alpha\in \Gamma\}$. Then it may be verified that $D$ is a compact set. There is a natural continuous surjection $\psi\colon D\to \Gamma$ which sends $b\in D$ to its symbolic coding $\alpha\in \Gamma$. By \cite[Remark 4.3.6]{Sriv}, there is a Cantor set $K\subset D$ such that $\psi|_K$ is injective. Then $\psi$ is a homeomorphism from $K$ onto $\Gamma':=\psi(K)$. Write $K=\{b_\alpha:\alpha\in \Gamma'\}$, where $\psi(b_\alpha)=\alpha$.

Now our aim is to check that $K$ is a syndetically $\eps$-scrambled set for $f$ for some $\eps>0$. Consider $b_\alpha\in K$ and $j\in \mathbb N$. First note that if $f^j(b_\alpha)\not\in [0,1/n)$, then $f^j(b_\alpha)$ must be contained in one of the intervals from the following two pieces of covering relations:
$H_{\alpha_{s}} \stackrel{f^{k_1}}{\Longrightarrow}L_2\stackrel{2,k_2}{\dashrightarrow} \ldots \stackrel{n,k_n}{\dashrightarrow} L_{n+1}$, or $L_{n+1} \stackrel{n,k_n}{\dashrightarrow} \ldots \stackrel{2,k_2}{\dashrightarrow} L_2\stackrel{f^{k_1}}{\Longrightarrow} H_{\alpha_{t}}$. This is because all other intervals in our infinite sequence of covering relations
are contained in $[0,\tfrac{1}{n+1}]$. In our infinite sequence of covering relations, the admissible combinations of the two pieces given above have bounded length since the interval $L_{n+2}$ should appear between two occurrences of $L_n$ or $H_0,H_1$. This shows that the set $\set{j: f^j(b_\alpha)\in [1/n,1]}$ is not thick, which in other words means that $\set{j: f^j(b_\alpha)\in [0,1/n)}$ is syndetic. Hence $(0,b_\alpha)\in \SProx(f)$. Since syndetically proximal relation is an equivalence relation (e.g. see \cite{Clay}), we conclude that any two elements of $K$ are syndetically proximal for $f$.

Without loss of generality we may assume that $H_0$ is to the left of $H_1$ inside the interval $[0,1]$. Denote $\eps=\tfrac{1}{2}\min\set{\dist(\set{0},H_0),\dist(\set{1},H_1),\dist(H_0,H_1)}.
$

If $b_\alpha,b_\beta\in K$ are distinct, then $\alpha_j\neq \beta_j$ for infinitely many $j$ and therefore we obtain $\limsup_{n\ra \infty} |f^{n}(b_\alpha)-f^{n}(b_\beta)|>\eps$. Thus the Cantor set $K$ is a syndetically $\epsilon$-scrambled set for $f$.

Next observe that $f^n(K)\subset H_0\cup H_1$ for infinitely many $n\in \mathbb N$. Moreover, if $U\subset [0,1]$ is any nonempty open set, then $f^n(U)\supset (\eps, 1-\eps)\supset H_0\cup H_1$ for all sufficiently large $n\in \mathbb N$. Therefore, for each nonempty open set $U\subset [0,1]$, there is $n\in \mathbb N$ such that $f^n(K)\subset f^n(U)$. Now by the proof of Lemma \ref{cdense}, we obtain a dense Mycielski set $S\subset [0,1]$ such that $S\cup \{0\}$ is a syndetically $\eps$-scrambled set for $f$.
\end{proof}

\begin{cor}
Let $f\colon [0,1]\to [0,1]$ be a transitive continuous map. Then there exist $\eps>0$, and a dense Mycielski syndetically $\eps$-scrambled set $S$ for $f$.
\end{cor}
\begin{proof}
If $f^2$ is transitive then it is well known that $f$ is mixing and we can apply Theorem~\ref{cdinvss}. If $f^2$ is not transitive, then there is $a\in (0,1)$ such that $a$ is the unique fixed point of $f$, the two intervals $[0,a]$, $[a,1]$ are cyclicly permuted by $f$, and both the maps $f^2|_{[0,a]}$, $f^2|_{[a,1]}$ are topologically mixing. We can apply Theorem~\ref{thm:mixna} to obtain $\eps>0$ and dense Mycielski sets $S_1\subset [0,a]$, $S_2\subset [a,1]$ such that $S_1\cup \{a\}$ and $S_2\cup \{a\}$ are syndetically $\eps$-scrambled sets for the maps $f^2|_{[0,a]}$, $f^2|_{[a,1]}$ respectively. Now it is not difficult to verify that $S_1\cup S_2$ is a dense Mycielski syndetically $\eps$-scrambled set for $f$ in $[0,1]$.
\end{proof}

\subsection{Interval maps with zero entropy}

In this section we examine the possibility of existence of syndetically scrambled sets for interval maps with zero entropy. It is known (e.g. see \cite{SLC}) that $\htop(f)=0$ for an interval map if and only if the period of any periodic point of $f$ is a power of 2. If only finitely many powers of 2 occur as periods of periodic points of $f$, then the results of \cite{SLC} imply that for every $x\in [0,1]$, there is a periodic point $p\in [0,1]$ such that $(x,p)\in \Asy(f)$, and therefore $f$ cannot have a scrambled set in this case. So it suffices to restrict our attention to interval maps of \emph{type $2^\infty$}, that is, those for which the set of periods of periodic points is equal to $\{1,2,2^2,2^3,\ldots \}$. Some interval maps of type $2^\infty$ posses scrambled sets and some others do not \cite{SmiOL}. For any interval map $f$ of type $2^\infty$, we will show that $\Prox(f)=\SProx(f)$ so that any scrambled set is automatically syndetically scrambled.

After \cite{SmiOL} we say that an interval $J\subset [0,1]$ is \emph{$f$-periodic of order
$k$} if $J$, $f(J), \ldots $, $f^{k-l}(J)$ are pairwise disjoint intervals and $f^k(J) = J$. Our main tool in the sequel will be the following important characterization of infinite $\omega$-limits sets of interval maps of type $2^\infty$.

\begin{lem}[{\cite[Theorem 3.5]{SmiOL}}]\label{lem:inf_ol_type2inf}
Let $f\colon [0,1]\to [0,1]$ be a continuous map of type $2^\infty$ and let $K\subset [0,1]$ be an infinite $\omega$-limit set of some point.
Then there is a sequence $( J_n)_{n=0}^\infty\subset [0,1]$ of $f$-periodic intervals with the following properties:
For any $n$,
\begin{enumerate}[(i)]
\item  $J_n$ has period $2^n$.

\item  $J_{n+l}\cup f^{2^n}(J_{n+1}) \subset J_n$.

\item $K\subset \bigcup_{j=0}^\infty f^j(J_n)$.

\item $K\cap f^j(J_n) \neq \emptyset$ for every $j\geq 0$.
\end{enumerate}
\end{lem}

The following result follows \cite[Prop.~4.3.]{JL}, since every set of Banach density $1$ must be syndetic. However, since the proof in our restricted case
is much shorter, we present it for completeness.

\begin{thm}\label{thm:int2inf}
Let $f\colon [0,1]\to [0,1]$ be a continuous map of type $2^\infty$. Then $\Prox(f)=\SProx(f)$.
\end{thm}
\begin{proof}
Consider $(x,y)\in \Prox(f)$. Since $\liminf_{n\ra\infty}|f^n(x)-f^n(y)|=0$, there is a point $z\in \omega(x,f)\cap \omega(y,f)$ with the property that $\lim_{j\ra \infty} |z-f^{n_j}(x)|=0$ and $\lim_{j\ra \infty} |z-f^{n_j}(y)|=0$ for some increasing sequence $(n_j)_{j=0}^\infty$.

If $z$ is an eventually periodic point, say of period $n$, then it cannot belong to any periodic interval of period larger than $n$.
It follows by Lemma~\ref{lem:inf_ol_type2inf} that in this case both the sets $\omega(x,f)$, $\omega(y,f)$ are finite. This implies the existence of some $0\le k\le n$ such that $(x,f^k(z)), (y,f^k(z))\in \Asy(f)$. Then $(x,y)\in \Asy(f)\subset \SProx(f)$.

For the second case, let us assume that $z$ is not eventually periodic, in particular $\orbp(z,f)$ as well as $\omega(x,f)$ are infinite.
Fix any $\eps>0$, and let $n$ be such that $1/2^n<\eps$. Let $J_n$ be the periodic interval provided by Lemma~\ref{lem:inf_ol_type2inf} for $\omega(x,f)$.
By the definition intervals $J_n,f(J_n),\ldots, f^{2^n-1}(J_n)$ are pairwise disjoint, thus one of them has diameter smaller than $\eps$ and so without loss of generality we may assume that $\diam [J_n]<\eps$. Since $z\in \omega(x,f)$ is a point with infinite orbit, there is $k\in \mathbb N$ such that $f^k(z)\in \Int J_n$. In particular, there is an open set $U\ni z$ such that $f^k(U)\subset J_n$. Then by the definition of $z$ we can find $s\in \mathbb N$ such that $f^s(x),f^s(y)\in U$. The proof is finished by the fact that
$
\set{i: d(f^i(x),f^i(y))<\eps}\supset \set{i: f^i(x)\in J_n,f^i(y)\in J_n}\supset \set{s+k+j2^n: j\in \N}.
$
\end{proof}

\begin{thm}\label{thm:interval_ent}
For a continuous $f\colon [0,1]\to [0,1]$, the following are equivalent:

\noindent (i) $\htop(f)>0$.

\noindent (ii) $\Prox (f)\setminus \SProx(f)\neq \emptyset$.

\noindent (iii) There is an $\eps$-scrambled set $S\subset [0,1]$ such that $S\times S\setminus \Delta \subset \Prox (f)\setminus \SProx(f)$. \end{thm}
\begin{proof}
We have already noted in the introduction of this section that if $h(f)=0$ and $f$ is not of type $2^\infty$, then every point is asymptotic to a periodic point, and hence $\Prox(f)=\Asy(f)=\SProx(f)$ in this case. If $h(f)=0$ and $f$ is of type $2^\infty$, then $\Prox(f)=\SProx(f)$ by Theorem~\ref{thm:int2inf}. This proves (ii) $\Longrightarrow$ (i).

It remains to prove (i) $\Longrightarrow$ (iii). By \cite[Theorem~9]{tksm} and by the fact that
$\SProx(f)\subset \SProx(f^n)$ it remains to show that there is an $\eps$-scrambled set $S\subset \Sigma^+_2$ such that $S\times S\setminus \Delta \subset \Prox (\sigma)\setminus \SProx(\sigma)$, where $(\Sigma^+_2,\sigma)$ is the one-sided full shift on two symbols.
Let $C\subset \Sigma^+_2$ be a scrambled set for $\sigma$. For any $x\in C$ define $z^x\in \Sigma^+_2$ by putting $z^x_0=0$ and
$z^x_i=x_n$ for all $i\in [4^n,4^{n+1})$ and $n=0,1,\ldots$. Let $S=\set{z^x: x\in C}$. It is easy to see that $S$ remains
a scrambled set. If $x,y\in C$ are distinct and if $x_n\neq y_n$, then $z^x_i\neq z^y_i$ for every $i\in [4^n,4^{n+1})$; this shows that the set $\set{i:z^x_i\neq z^y_i}$ is thick, and thus $(z^x,z^y)\notin \SProx(\sigma)$.
\end{proof}

It was proved in \cite[Thm.~5.5]{JL} that for interval maps $\htop(f)>0$ iff $\Prox(f)$ is not an equivalence relation (a few more equivalent conditions in terms of Li-Yorke tuples are provided there as well).
In particular, implication $(i)\Longrightarrow (ii)$ in Theorem~\ref{thm:interval_ent} follows directly form these results. It is also noteworthy that Wu proved in \cite{Wu} that $\Prox(f)$ is an equivalence relation exactly when $\SProx(f)=\Prox(f)$ (while in the paper \cite{Wu} only actions of groups are considered, the proof works for the non-invertible case as well). This combined with \cite{JL} can serve also as a proof of converse implication.

It is not hard to verify that Theorem~\ref{thm:interval_ent} deeply relies on the special structure of the underlying space. For a general compact dynamical system, there are no implications between the following two properties: (i) $\htop(f)>0$, (ii) $\Prox (f)\setminus \SProx(f)\ne \emptyset$. If we consider a proximal system with positive entropy (for example, the one constructed in \cite{WMProx}), then we see that (i) does not imply (ii). If we consider a non-trivial minimal weakly mixing system with zero entropy (for instance, the Chac\'{o}n flow; see \cite{BK} for a discussion) then we see that it cannot be proximal, thus $\SProx(f)$ is nowhere dense, while $\Prox(f)$ is residual in $X^2$, and thus (ii) does not imply (i).

\section*{Acknowledgements}

The authors are grateful to Jian Li for his valuable comments which leaded to an essential improvement of the present paper.
We express many thanks to a referee of this paper for his careful reading and numerous suggestions.

Research of Piotr Oprocha leading to results contained in this paper was also supported by the Marie Curie European Reintegration Grant of the European Commission under grant agreement no. PERG08-GA-2010-272297.

\bibliographystyle{amsplain}
\bibliography{SyProxMO}
\end{document}